\documentclass[journal,onecolumn,draftcls]{IEEEtran}
\usepackage{amsmath}
\usepackage{amssymb}
\usepackage{amsfonts}
\usepackage{epsfig}
\usepackage[utf8]{inputenc}
\usepackage[english]{babel}
\usepackage{amsthm}
\usepackage{algorithm}
\usepackage{algpseudocode}
\usepackage[dvipsnames]{xcolor}
\usepackage{authblk}
\usepackage{cite}

\usepackage{subfigure}
\usepackage{dsfont}

\def\l{\left}
\def\r{\right}
\newcommand{\expect}[1]{\mathbb{E}{\l[#1\r]}}
\newcommand{\dotp}[2]{\left\langle#1,#2\right\rangle}
\newcommand{\mc}[1]{\mathcal{#1}}
\newcommand{\mf}[1]{\mathbf{#1}}
\newcommand{\mb}{\mathbb}
\newcommand{\tc}{\textcolor}

\newcommand{\DefinedAs}[0]{\mathrel{\mathop:}=}
\DeclareMathOperator*{\minimize}{minimize}
\DeclareMathOperator*{\maximize}{maximize}
\DeclareMathOperator*{\argmin}{argmin} 
\DeclareMathOperator*{\argmax}{argmax}

\definecolor{figcolor1}{RGB}{33,255,255}

\theoremstyle{remark}
\usepackage[figurename=Fig.]{caption}

\newtheorem{theorem}{Theorem}

\newtheorem{lemma}[theorem]{Lemma}

\newtheorem{definition}{Definition}
\newtheorem{remark}{Remark}

\newcommand{\enma}[1]   {\ensuremath{#1}}

\newcommand{\non}{\nonumber}

\newcommand{\beq}{\begin{equation}}
\newcommand{\eeq}{\end{equation}}
\newcommand{\bseq}{\begin{subequations}}
\newcommand{\eseq}{\end{subequations}}
\newcommand{\beqn}{\begin{eqnarray}}
\newcommand{\eeqn}{\end{eqnarray}}
\newcommand{\ba}{\begin{array}}
\newcommand{\ea}{\end{array}}
\newcommand{\bct}{\begin{center}}
\newcommand{\ect}{\end{center}}
\newcommand{\btmz}{\begin{itemize}}
\newcommand{\etmz}{\end{itemize}}
\newcommand{\benum}{\begin{enumerate}}
\newcommand{\eenum}{\end{enumerate}}

\newcommand{\norm}[1]{\| #1 \|}                 
\newcommand{\Norm}[1]{\left\| #1 \right\|}

\newcommand{\diag}      {\enma{\mathrm{diag}}}

\newcommand{\be}{\begin{equation}}
\newcommand{\ee}{\end{equation}}

\newcommand{\cplxs}{ C\kern -.35em \rule{0.03 em}{.7 ex}~   }

\def\complex{\hbox{C\kern -.45em \rule{0.03 em}{1.5 ex}}~}

\newcommand{\bi}{\begin{itemize}}
\newcommand{\ei}{\end{itemize}}

\newtheorem{assumption}{Assumption}

\renewcommand{\baselinestretch}{1.38} 

\makeatletter
\def\endthebibliography{%
	\def\@noitemerr{\@latex@warning{Empty `thebibliography' environment}}%
	\endlist
}
\makeatother

\begin{document}
	
%
\title{\huge Fast Multi-Agent Temporal-Difference Learning via \\[-0.15cm] Homotopy Stochastic Primal-Dual Optimization}
%
%
%

\author{
Dongsheng~Ding,~Xiaohan~Wei,~Zhuoran~Yang,~Zhaoran~Wang, and Mihailo~R.~Jovanovi\'{c}
\thanks{The work of D.\ Ding and M.\ R.\ Jovanovi\'{c} is supported in part by the National Science Foundation under awards ECCS-1708906 and 1809833.}
\thanks{
D.\ Ding, X.\ Wei and M.\ R.\ Jovanovi\'{c} are with the Ming Hsieh Department of Electrical and Computer Engineering, University of Southern California, Los Angeles, CA 90089.
Z.\ Yang is with the Department of Operations Research and Financial Engineering, Princeton University, Princeton, NJ 08544.
Z.\ Wang is with the Department of Industrial Engineering and Management Sciences, Northwestern University, Evanston, IL 60208. {E-mails: dongshed@usc.edu, xiaohanw@usc.edu, zy6@princeton.edu, zhaoran.wang@northwestern.edu, mihailo@usc.edu.}
}

}

\maketitle

	\vspace*{-6ex} 
\begin{abstract}
	We study the policy evaluation problem in multi-agent reinforcement learning where a group of agents, with jointly observed states and private local actions and rewards, collaborate to learn the value function of a given policy via local computation and communication \tc{black}{over a connected undirected network.} This problem arises in various large-scale multi-agent systems, including power grids, intelligent transportation systems, wireless sensor networks, and multi-agent robotics. When the dimension of state-action space is large, the temporal-difference learning with linear function approximation is widely used. In this paper, we develop a new distributed temporal-difference learning algorithm \tc{black}{and quantify its finite-time performance. Our algorithm combines a distributed stochastic primal-dual method with a homotopy-based approach to adaptively adjust the learning rate in order to minimize the mean-square projected Bellman error by taking fresh online samples from a causal on-policy trajectory. We explicitly take into account the Markovian nature of sampling and improve the best-known finite-time error bound from $O(1/\sqrt{T})$ to~$O(1/T)$, where $T$ is the total number of iterations.} 
\end{abstract}

	\vspace*{-2ex}
\begin{IEEEkeywords}
Multi-agent reinforcement learning, distributed temporal-difference learning, stochastic saddle point problems, primal-dual method, homotopy method. 
\end{IEEEkeywords}

%
\IEEEpeerreviewmaketitle

	\vspace*{-2ex}
\section{Introduction}

	Temporal-difference (TD) learning is a central approach to policy evaluation in modern reinforcement learning (RL)~\cite{sutton2018reinforcement}. It was introduced in~\cite{sutton1988learning, bertsekas1996neuro, baird1995residual} and significant advances have been made in a host of single-agent decision-making applications, including Atari or Go games~\cite{mnih2015human,silver2016mastering}. Recently, TD learning has been used to address multi-agent decision making problems for large-scale systems, including power grids~\cite{misra2013residential}, intelligent transportation systems~\cite{kuyer2008multiagent}, wireless sensor networks~\cite{pennesi2010distributed}, and multi-agent robotics~\cite{kober2013reinforcement}. Motivated by these applications, we introduce an extension of TD learning to a distributed setting of policy evaluation. This setup involves a group of agents that communicate over a connected undirected network. All agents share a joint state and dynamics of state transition are governed by the local actions of agents which follow a local policy and own a private local reward. To maximize the total reward, i.e., the sum of all local rewards, it is essential to quantify performance that each agent achieves if it follows a particular policy \tc{black}{while interacting with the environment and using only local data and information exchange with its neighbors.} This task is usually referred to as a distributed policy evaluation problem and it has received significant recent attention~\cite{macua2014distributed,mathkar2016distributed,lee2018primal,wai2018multi,cassano2018multi,doan2019convergence,doan2019finite,sun2019finite,sha2020asynchronous}.

	\vspace*{-2ex}
\subsection{Related work}

TD learning with linear function approximators is a popular approach for estimating the value function for an agent that follows a particular policy. The asymptotic convergence of the original TD method, \tc{black}{which is known as TD(0),} and its extension TD($\lambda$) was established in~\cite{tsitsiklis1997analysis}. \tc{black}{In spite of their wide-spread use}, these algorithms can become unstable and convergence cannot be guaranteed in off-policy learning scenarios~\cite{baird1995residual,szepesvari2010algorithms}. To ensure stability, batch methods, e.g., Least-Squares Temporal-Difference (LSTD) learning~\cite{bradtke1996linear}, have been proposed at the expense of increased computational complexity. To achieve both stability and low computational cost, a class of gradient-based temporal-difference (GTD) algorithms~\cite{sutton2009convergent,sutton2009fast}, e.g., GTD, GTD2, and TDC, were proposed and their asymptotic convergence in off-policy settings was established by analyzing certain ordinary differential equations (ODEs)~\cite{borkar2000ode}. However, these are not true stochastic gradient methods with respect to the original objective functions~\cite{szepesvari2010algorithms} because the underlying TD objectives, e.g., mean square Bellman error (MSBE) in GTD or mean square projected Bellman error (MSPBE) in GTD2, involve products and inverses of expectations. As such, these cannot be sampled directly and it is difficult to analyze efficiency with respect to the TD objectives.

The \tc{black}{finite-time or finite-sample performance} analysis of TD algorithms is critically important in applications \tc{black}{with limited time budgets and finite amount of data.} Early results are based on the stochastic approximation approach under i.i.d.\ sampling. For TD(0) and GTD, \tc{black}{$O(1/T^{\alpha})$ error bound with $\alpha\in(0,0.5)$ was established in~\cite{dalal2018finite, dalal2017finite} and an improved $O(1/T)$ bound was provided in~\cite{lakshminarayanan2018linear}, where $T$ is the total number of iterations. In the Markov setting, $O (1/T)$ bound was established for TD(0) that involves a projection step~\cite{bhandari2018finite}; for a linear stochastic approximation algorithm driven by Markovian noise~\cite{srikant2019finite}; and for an on-policy TD algorithm known as SARSA~\cite{zou2019finite}.} Recent work~\cite{hu2019characterizing} provides complementary \tc{black}{analysis of} TD algorithms using the Markov jump linear system theory. To enable both on- and off-policy implementation, an optimization-based approach~\cite{liu2015finite} was used to cast MSBPE into a convex-concave objective that allows \tc{black}{the use of} stochastic gradient algorithms~\cite{nemirovski2009robust} with $O(1/\sqrt T)$ bound for i.i.d.\ samples; \tc{black}{an extension of this approach} to the Markov setting was provided in~\cite{wang2017finite}. In~\cite{touati2017convergent}, the \tc{black}{finite-time error bound} of GTD was improved to $O(1/T)$ for i.i.d.\ sampled data but it remains unclear how to extend these results to multi-agent scenarios where data is not only Markovian but also distributed over a network.

In the context of distributed policy evaluation, several attempts have been made to extend TD algorithms to a multi-agent setup using linear function approximators. {\color{black}When the reward is global and actions are local, mean square convergence of a distributed primal-dual GTD algorithm for minimizing MSPBE, with diffusion updates, was established in~\cite{macua2014distributed} and an extension to time-varying networks was made in~\cite{stankovic2016multi}.} In~\cite{mathkar2016distributed}, the authors combine the gossip averaging scheme~\cite{boyd2005gossip} with TD(0) and show asymptotic convergence. In the off-line setting, references~\cite{wai2018multi} and~\cite{cassano2018multi} propose different consensus-based primal-dual algorithms for minimizing a batch-sampled version of MSPBE with linear convergence; a fully asynchronous gossip-based extension was studied in~\cite{sha2020asynchronous} and its communication efficiency was analyzed  in~\cite{ren2021communication}. To understand/gain data efficiency, the recent focus of multi-agent TD learning research has shifted to \tc{black}{finite-time or finite-sample performance analysis}. \tc{black}{For distributed TD(0) and TD($\lambda$) with local rewards, $O(1/T)$ error bound was established  in~\cite{doan2019convergence} and~\cite{doan2019finite}, respectively. Linear convergence of distributed TD(0) to a neighborhood of the stationary point was proved in~\cite{sun2019finite} and $O(1/\sqrt T)$ error bound for distributed GTD was shown in~\cite{lee2018primal}. In~\cite{doan2019finited}, $O(1/T^{2/3})$ error bound was provided for a distributed variant of two-time-scale stochastic approximation algorithm.} Apart from~\cite{doan2019finite,sun2019finite}, other finite sample results rely on the i.i.d.\ state sampling in policy evaluation. In most RL applications, this assumption is overly restrictive because of the Markovian nature of state trajectory samples. In~\cite{gyorfi1996averaged}, an example was provided to demonstrate that i.i.d.\ sampling-based convergence \tc{black}{guarantees can fail to hold} when samples become correlated. It is thus relevant to examine how to design an online distributed learning algorithm for the policy evaluation problem (e.g., MSPBE minimization) in the Markovian setting. Such distributed learning algorithms are essential in multi-agent RL; e.g., see the distributed variant of policy gradient theorem~\cite{zhang2018fully} along with recent surveys~\cite{zhang2019multi,zhang2019decentralized,lee2019optimization}.

		\vspace*{-2ex}
\subsection{Our contributions}

	Our paper contains two main contributions. The first contribution is to provide a new distributed TD learning algorithm for \tc{black}{the multi-agent policy evaluation problem that only takes fresh online samples from a causal on-policy trajectory. We refer to our algorithm as a distributed homotopy stochastic primal-dual method because it maintains stochastic primal-dual updates over the network while employing homotopy-based adaptation of the learning rate in each restarting round.} Our approach \tc{black}{combines features of} distributed dual averaging~\cite{duchi2012dual} and homotopy-based~\cite{xiao2013proximal} approaches. \tc{black}{In contrast to prior results~\cite{lee2018primal,lee2018stochastic}}, we show that the optimal \tc{black}{finite-time error bound} $O(1/T)$ for stochastic convex optimization~\cite{agarwal2009information} is achieved and we explicitly characterize the influence of network size and topology \tc{black}{on finite-time performance}. The second contribution is an ergodic analysis that explicitly takes into account Markovian nature of samples. Our analysis \tc{black}{extends} ergodic mirror descent~\cite{duchi2012ergodic} to stochastic primal-dual algorithms, \tc{black}{departs from i.i.d.\ restriction~\cite{wai2018multi,doan2019convergence,doan2019finited},} and addresses a broader class of problems than the existing literature. 

	We formulate the multi-agent TD learning as the minimization problem of mean square projected Bellman error (MSPBE). \tc{black}{We employ Fenchel duality} to cast the MSPBE minimization as a stochastic saddle point problem where the primal-dual objective is convex and strongly-concave \tc{black}{with respect to primal and dual variables, respectively.} Since the primal-dual objective has a linear dependence on expectations, we can obtain unbiased estimates of gradients from state samples \tc{black}{thereby overcoming a challenge that approaches based on naive TD objectives face~\cite{liu2015finite}.} This allows us to design a true stochastic primal-dual learning algorithm and perform the \tc{black}{finite-time performance analysis} in Markov setting~\cite{nemirovski2009robust,duchi2012ergodic}. \tc{black}{Our primal-dual formulation utilizes distributed dual averaging~\cite{duchi2012dual} and for our homotopy-based distributed learning algorithm we establish a sharp finite-time error bound} in terms of network size and topology. This differentiates our work from the approaches and results in~\cite{lee2018primal,doan2019convergence,doan2019finite}. To the best of our knowledge, we are the first to utilize the homotopy-based approach for solving a class of distributed convex-concave saddle point programs with \tc{black}{$O(1/T)$ finite-time performance bound.} 

		\vspace*{-2ex}
\subsection{Paper outline} 

Our presentation is organized as follows. In Section~\ref{sec.probl}, \tc{black}{we introduce a class of multi-agent stochastic saddle point problems that contain, as a special instance, minimization of a mean square projected Bellman error via distributed TD learning.} In Section~\ref{sec.alg}, we \tc{black}{develop a homotopy-based} online distributed primal-dual algorithm to solve this problem and \tc{black}{establish a finite-time performance bound for} the proposed algorithm. In Section~\ref{sec.convergence}, we \tc{black}{prove the main result,} in Section~\ref{sec.comp}, we offer computational experiments to demonstrate the merits and the effectiveness of our theoretical findings and, in Section~\ref{sec.concl}, we close the paper with concluding remarks.

		\newpage
\section{Problem \tc{black}{formulation and background}}
\label{sec.probl}

\tc{black}{In this section, we formulate a multi-agent stochastic saddle point problem over a connected undirected network. The motivation for studying this class of problems comes from distributed reinforcement learning where a group of agents with jointly observed states and private local actions/rewards collaborate to learn the value function of a given policy via local computation and communication. We exploit the structure of the underlying optimization problem to demonstrate that it enables unbiased estimation of the saddle point objective from Markovian samples. Furthermore, we discuss an algorithm that is convenient for distributed implementation and finite-time \mbox{performance analysis.}}
		
\vspace*{-2ex}
\subsection{\color{black}Multi-agent stochastic optimization problem}
\label{sec.SP}

{\color{black}We consider a stochastic optimization problem over a connected undirected network $\mc G = (\mc V, \mc E)$ with $N$ agents,
\begin{subequations}\label{eq.single}
\begin{equation}
\label{eq.maso}
\minimize\limits_{x \, \in \, \mc X} \; 
	\dfrac{1}{N} \displaystyle \sum_{j \, = \, 1}^{N} f_j(x)
\end{equation}
where $\mc V = \{1,\ldots,N\}$ is the set of nodes, $\mc E \subset \mc V\times \mc V$ is the set of edges, $x$ is the optimization variable, $\mathcal{X} \subset \mathbb{R}^d$ is a compact convex set, and $f_j$: $\mathcal{X}\to\mathbb{R}$ is a local objective function determined by,
\begin{equation}
	f_j(x)  
	\, =  \, 
	\max_{y_j \, \in \, \mathcal Y} 
	~ 
	\underbrace{\mathbb E_{\xi\sim\Pi}[\,\Psi_j(x, y_j; \xi)\,]}_{\psi_j (x, y_j)}.
\label{eq.maso_local}
\end{equation}
\end{subequations}
Here, $y_j$ is a local variable that belongs to a convex compact set $\mathcal Y\subset \mathbb{R}^d$ and $\Psi_j(x, y_j; \xi)$ is a stochastic function of a random variable $\xi$ which is distributed according to the stationary distribution $\Pi$ of a Markov chain. Equivalently, problem~\eqref{eq.single} can be cast as a multi-agent stochastic saddle point problem,
\be
\minimize\limits_{ x \, \in \, \mathcal X} 
\;
\maximize\limits_{ y_j \, \in \, \mathcal Y} 
\;
\dfrac{1}{N} \displaystyle \sum_{j \, = \, 1}^{N} \psi_j (x, y_j) 
\label{eq.main}
\ee
with the primal variable $x$ and the dual variable $y \DefinedAs (y_1,\ldots,y_N)$. 

Each agent $j$ can only communicate with its neighbors over the network $\mathcal{G}$ and receive samples $\xi$ from a stochastic process that converges to the stationary distribution $\Pi$. Optimal solution $(x^\star,y^\star)$ to the saddle point problem~\eqref{eq.main} satisfies
\begin{equation*}
\begin{array}{rcl}
x^\star 
& \!\! \DefinedAs \!\! & 
\argmin\limits_{x \, \in \, \mc X} ~ \dfrac{1}{N}\displaystyle\sum_{j \, = \, 1}^{N}\psi_j(x, y_{j}^\star)
\\[0.25cm]
y_j^\star 
& \!\! \DefinedAs \!\! &
\argmax\limits_{y_j \, \in \, \mc Y} ~ \psi_j(x^\star, y_{j})
\end{array}
\end{equation*}
and the primary motivation for studying this class of problems comes from multi-agent reinforcement learning where a group of agents with jointly observed states and private local actions/rewards collaborate to learn the value function of a given policy via local computation and communication. Although our theory and algorithm can be readily extended to other settings, in this paper we restrict our attention to the Markovian structure in the context of policy evaluation. Formulation~\eqref{eq.single} arises in a host of large-scale multi-agent systems, e.g., in supervised learning~\cite{ying2016stochastic} and in nonparametric regression~\cite{dikkala2020minimax}, and we describe it next.}

		\vspace*{-2ex}
\subsection{Multi-agent Markov decision process}\label{sec.MDP}

{\color{black}Let us consider a control system described by a Markov decision process (MDP) over a connected undirected network $\mc G$ with $N$ agents, the state space $\mathcal{S}$, and the joint action space $\mathcal A \DefinedAs \mathcal{A}_1 \times \cdots \times \mathcal{A}_N$.} Without loss of generality we assume that, for each agent $j$, the local action space $\mathcal A_j$ is the same for all states. Let $\mathcal{P}^a = [\,\mathcal{P}_{s,s'}^a\,]_{s,s'\in\mathcal{S}}$ be the probability transition matrix under a joint action $a = (a_1,\ldots, a_N )\in \mathcal A$, where $\mathcal{P}_{s,s'}^a$ is the transition probability from state $s$ to state $s'$ and let $\mathcal{R}_j(s,a)$ be the local reward received by agent $j$ that corresponds to the pair $(s,a)$. The multi-agent MDP can be represented by the tuple,
\be
\label{eq.mdp}
\left(\,
{\color{black}\mc S}
,\, \{\mathcal{A}_j \}_{j=1}^N,\, \mathcal{P}^a,\, \{\mathcal{R}_j \}_{j=1}^N,\, \gamma\,\right)
\ee
where $\gamma\in(0,1)$ is a fixed discount factor. 

When the state, actions, and rewards {are all} globally observable, the multi-agent MDP simplifies to a single-agent MDP. In many applications (e.g., see~\cite{kober2013reinforcement,misra2013residential,lee2018primal}), both the actions $a_j\in\mathcal{A}_j$ and the rewards $\mathcal{R}_j(s,a)$ are private and every agent can only communicate with its neighbors over the network $\mc G$. It is thus critically important to extend single-agent TD learning algorithms to a setup in which only local information exchange is available. 

\tc{black}{We consider a cooperative learning task in which agents aim to maximize the total reward $(1/N)\sum_{j} \mathcal{R}_j(s,a)$ and, in Fig.~\ref{fig.MARL}, we illustrate the interaction between the environment and the agents.} Let $\pi$: $\mathcal{S}\times\mathcal{A}\rightarrow [0,1]$ be a joint policy which specifies the probability to take an action $a\in\mathcal{A}$ at state $s\in\mathcal{S}$. We define the global reward $R^{\pi}(s)$ at state $s\in\mathcal S$ under policy $\pi$ to be the expected value of the average of all local rewards, 
\begin{equation}
	\label{eq.globalreward}
	\textstyle
	R^{\pi}(s) \,=\, \dfrac{1}{N} \displaystyle\sum_{j \, = \, 1}^{N} R_j^{\pi}(s)
\end{equation}
where $R_j^{\pi}(s) \DefinedAs \mathbb{E}_{a \, \sim \, \pi (\cdot|s)}\left[\,\mathcal{R}_j(s,a)\, \right]$.

For any fixed joint policy $\pi$, the multi-agent MDP becomes a Markov chain over $\mathcal{S}$ with the probability transition matrix $\mf P^{\pi} \in \mathbb{R}^{|\mathcal{S}| \times |\mathcal{S}|}$, where
	$
\mf P_{s,s'}^{\pi} = \sum_{a\in\mathcal{A}} \pi(a\,|\,s) \mathcal{P}_{s,s'}^a
	$
\tc{black}{is the $(s,s')$-element of $\mf P^{\pi}$. If the Markov chain associated with the policy $\pi$ is aperiodic and irreducible then, for any initial state, it converges to the unique stationary distribution $\Pi$ with a geometric rate~\cite{levin2017markov}; see Assumption~\ref{as.mc_general} for a formal statement.} 

\begin{figure}[]
	\centering
	\includegraphics[width=7.8cm,height=4.9cm]{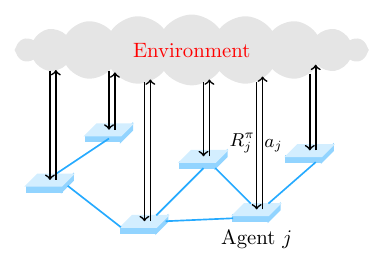}
	\caption{A system with six agents that communicate over a connected undirected network. Each agent interacts with the environment by receiving a private reward and taking a local action. }
	\label{fig.MARL}
\end{figure}

	\vspace*{-2ex}
\subsection{\color{black}Multi-agent policy evaluation and TD learning}
\label{sec.TD}

Let the value function of a policy $\pi$, $V^{\pi}$: $\mathcal{S}\rightarrow \mathbb{R}$, be defined as the expected value of discounted cumulative rewards,
\begin{equation*}
	\label{eq.value}
	\textstyle
	V^{\pi} (s) \,=\,\mathbb{E} \left[\, \displaystyle\sum_{p \, = \, 0}^{\infty} \gamma^p \mathcal{R}^{\pi} (s_p) \,\Big\vert\, s_0=s, \; \pi \,\right]
\end{equation*}
where $s_0=s$ is the initial state. 
If we arrange $V^{\pi}(s)$ and $R^{\pi}(s)$ over all states $s\in\mathcal{S}$ into the vectors $\mathbf{V}^{\pi}$ and $\mathbf{R}^{\pi}$, the Bellman equation for $\mathbf{V}^{\pi}$ can be written as~\cite{puterman2014markov},
\begin{equation}
	\label{eq.bellman}
	\mathbf{V}^{\pi} 
	\, =  \, 
	\gamma \, \mathbf{P}^{\pi} \mathbf{V}^{\pi} \, + \, \mathbf{R}^{\pi} .
\end{equation}
Since it is challenging to evaluate $\mathbf{V}^{\pi}$ directly for a large state space, we approximate $\mathbf{V}^{\pi}$ using a family of linear functions 
$
\{
V_{x} (s) ={\phi}^T (s)\, {x},\, x\in \mathbb{R}^d
\}
$,
where $x\in\mathbb{R}^d$ is the vector of unknown parameters and $\phi$: $\mathcal{S}\rightarrow \mathbb{R}^d$ is a known dictionary consisting of $d$ features. If we arrange $\{V_{x}(s)\}_{s \,\in\, \mathcal{S}}$ into the vector $\mathbf{V}_{x}\in\mathbb{R}^{|\mathcal{S}|}$, we have $\mathbf{V}_{x} = \Phi\, x$ where the $i$th row of the matrix $\Phi \in \mathbb{R}^{|\mathcal{S}|\times d}$ is given by $\phi^T (s_i)$. Since the dictionary is a function \tc{black}{determined by}, e.g., polynomial basis, it is not restrictive \tc{black}{to assume that the matrix $\Phi$} has the full column rank~\cite{bertsekas1996neuro}.

The goal of policy evaluation now becomes to determine the \tc{black}{vector of} feature weights $x\in\mathbb{R}^d$ so that $\mathbf{V}_x$ approximates the true value function $\mathbf{V}^\pi$. The objective of a typical TD learning method is to minimize the mean square Bellman error (MSBE)~\cite{sutton2009convergent}, 
	\[
	\tfrac{1}{2} \, \Vert \mathbf{V}_{x} - \gamma\, \mathbf{P}^{\pi} \mathbf{V}_{x} - \mathbf{R}^{\pi}\Vert^2_{ D}
	\]
where $ D \DefinedAs \mathrm{diag} \{ \Pi(s),s\in\mathcal{S} \}\in\mathbb{R}^{|\mathcal{S}|\times|\mathcal{S}|}$ is a diagonal matrix determined by the stationary distribution $\Pi$. As discussed in~\cite{sutton2009fast}, the solution to the fixed point problem
$ \mathbf{V}_{x} = \gamma \mathbf{P}^{\pi} \mathbf{V}_{x} + \mathbf{R}^{\pi}$ may not exist because the right-hand-side may not stay in the column space of the matrix $\Phi$. To address this challenge, GTD algorithm~\cite{sutton2009fast} was proposed to minimize the mean square projected Bellman error (MSPBE),
	\[
	f(x)
	\, \DefinedAs \,
	\tfrac{1}{2} 
	\Vert  
	P_{\Phi} 
	( \mathbf{V}_{x} - \gamma \mathbf{P}^{\pi} \mathbf{V}_{x} - \mathbf{R}^{\pi} )
	\Vert^2_{ D}
	\]
via stochastic-gradient-type updates, where $P_{\Phi} \DefinedAs \Phi(\Phi^T D \Phi)^{-1}\Phi^T D$ is a projection operator onto the column space of $\Phi$. \tc{black}{Equivalently, MSPBE is given by
	$
	f(x) 
	= 
	\tfrac{1}{2} 
			\left\Vert 
			\Phi^T D 
			\left( 
			\mathbf{V}_{x}  
			- 
			\gamma\, \mathbf{P}^{\pi} \mathbf{V}_{x} 
			- 
			\mathbf{R}^{\pi} 
			\right)  
			\right\Vert^2_{\left( \Phi^T 	D \Phi \right)^{-1}}
	$
and it can be compactly written as,
	\begin{subequations}
	\label{eq.mspbe}
	\begin{equation}
			f (x) 
			\, = \, 
			\tfrac{1}{2} \left\Vert  Ax \, - \, b\right\Vert^2_{ C^{-1}} 
			\label{eq.f}
	\end{equation} 
where $A$, $b$, and $C$ are obtained by taking expectations over the stationary distribution $\Pi$,
		\begin{equation}
		\begin{array}{rcl}
			A 
			& \!\! \DefinedAs \!\! & 
			\mathbb{E}_{s\,\sim\,\Pi} [ \, \phi(s) (\phi(s) -\gamma\phi(s') )^T \,]
			\\[0.cm]
			b 
			& \!\! \DefinedAs \!\! & 
			\mathbb{E}_{s\,\sim\,\Pi} [ \, R^{\pi}(s)\phi(s)\, ]
			\\[0.cm] 
			C  
			& \!\! \DefinedAs \!\! &  
			\mathbb{E}_{s\,\sim\,\Pi} [ \,\phi(s)\phi(s)^T \,].
		\end{array}
		\label{eq.AbC}
	\end{equation}
	\end{subequations}
\begin{assumption}\label{as.AC}
	There exists a feature matrix $\Phi$ such that the matrices $A$ and $C$ are full rank and positive definite, respectively.
\end{assumption}}
In~\cite[page~300]{bertsekas1996neuro}, it was shown that the full column rank matrix $\Phi$ yields a full rank $A$ and a positive definite $C$ and that the objective function $f$ in~\eqref{eq.mspbe} has a unique minimizer. {\color{black}Nevertheless, when $A$, $b$, and $C$ are replaced by their sampled versions it is challenging to solve~\eqref{eq.mspbe} because their nonlinear dependence on the underlying samples introduces bias in the objective function $f$. In what follows, we address the sampling challenge by reformulating~\eqref{eq.mspbe} in terms of a saddle-point objective. }	

\tc{black}{Since the global reward $R^{\pi}(s)$ in~\eqref{eq.globalreward} is determined by the average of all local rewards $R_j^{\pi}(s)$, we can express the vector $b$ as $b = (1/N)\sum_{j}  b_j$, where $ b_j \DefinedAs \mathbb{E}_{s \sim \Pi} [\, R_j^{\pi}(s)\phi(s)\,]$. Thus, the problem of minimizing MSPBE~\eqref{eq.mspbe} can be cast as
\begin{equation}
	\label{eq.mspbesum}
	\minimize\limits_{x \, \in \, \mc X} 
	\;  
	\dfrac{1}{N} \displaystyle \sum_{j \, = \, 1}^{N} 	
	\dfrac{1}{2} \left\Vert  Ax \, - \,  b_j\right\Vert^2_{C^{-1}}
\end{equation}
where $f_j (x) \DefinedAs \tfrac{1}{2} \left\Vert  Ax-  b_j\right\Vert^2_{C^{-1}}$ is the local MSPBE for the agent $j$ and $\mc X \subset \mathbb{R}^d$ is a compact convex set that contains the unique minimizer of $f (x) = (1/N)\sum_{j}  f_j (x)$.}	

A decentralized stochastic optimization problem~\eqref{eq.mspbesum} with $N$ private stochastic objectives involves products and inverses of the expectations; \tc{black}{cf.~\eqref{eq.AbC}}. This unique feature of MSPBE is not encountered in typical distributed optimization settings~\cite{nedic2009distributed,duchi2012dual} and it makes the problem of obtaining an unbiased estimator of the objective function from a few state samples challenging. 

\tc{black}{Using Fenchel duality, we can express each local MSPBE in~\eqref{eq.mspbesum} as}
\begin{equation}
	f_j (x) 
	\, = \, 
	\max_{y_j \, \in \, \mathcal Y} \left( y_j^T(Ax-  b_j)  - \tfrac{1}{2} \, y_j^TCy_j \right)
	\label{eq.fj-new}
\end{equation}
where $y_j$ is a dual variable and $\mathcal Y\subset \mathbb{R}^d$ is a convex compact set such that $C^{-1}(Ax-  b_j)\in\mathcal Y$ for all $x\in\mathcal X$. Since $C$ is a positive definite matrix and $\mathcal X$ is a compact set, such $\mathcal Y$ exists. In fact, one could take a ball centered at the origin with a radius greater than {$(1/\lambda_{\min}(C)) \sup \|Ax-b_j\|$, where the supremum is taken over $x\in\mathcal X$ and $j \in \{1,\ldots,N \}$}. Thus, we can reformulate~\eqref{eq.mspbesum} as a decentralized stochastic saddle point problem~\eqref{eq.main} with $\psi_j (x, y_j) = y_j^T(Ax-  b_j)  - \tfrac{1}{2} \, y_j^TCy_j$. \tc{black}{By replacing expectations in the expressions for $A$, $b_j$, and  $C$ with their samples that arise from the stationary distribution $\Pi$, we obtain an unbiased estimate of the saddle point objective $\psi_j$. We also note that each agent $j$ indeed takes a local MSPBE as its local objective function $f_j(x) = \tfrac{1}{2} \left\Vert  Ax-  b_j\right\Vert^2_{C^{-1}}$.}

\tc{black}{Since the stationary distribution is not known, it is not possible to directly estimate $A$, $b$, and $C$. However, as we explain next, the policy evaluation problem allows correlated sampling according to a Markov process.}

	\vspace*{-2ex}
\subsection{Standard stochastic primal-dual algorithm}
\label{sec.spd}

When i.i.d.\ samples from the stationary distribution $\Pi$ are available and a centralized controller exists, the stochastic approximation method can be used to compute the solution to~\eqref{eq.main} with a convergence rate $O(1/\sqrt T)$ in terms of the primal-dual gap~\cite{nemirovski2009robust}. The stochastic primal-dual algorithm generates two pairs of {vectors $(x'(t),y'(t))$ and $(x(t),y(t))$} that are contained in $\mc X \times\mc Y^N$, where $t$ is a positive integer. At iteration $t$, the primal-dual updates are given by
\be
\label{eq.mrsa}
\ba{rcl}
x'(t+1)  & \!\!  = \!\!  & x'(t) \, - \, \eta(t) \,G_x(x(t),y(t);\xi_t)
\\[0.1cm]
x(t+1) & \!\!  = \!\!  & \mc P_{\mc X}( \,x'(t+1) \,) 
\\[0.1cm]
y'(t+1) & \!\!  = \!\!  & y'(t) \, + \, \eta(t) \,G_y(x(t),y(t);\xi_t)
\\[0.1cm]
 y(t+1) & \!\!  = \!\!  &\mc P_{\mc Y^N}( \, y'(t+1) \,)
\ea
\ee
where $\eta(t)$ is a non-increasing sequence of stepsizes, \tc{black}{$\mathcal P_{\mc X} (x') \DefinedAs \argmin_{x\,\in \, \mc X} \Norm{x - x'}$ and $\mathcal P_{\mc Y^N} (y') \DefinedAs \argmin_{y\,\in \, \mc Y^N}$ $\Norm{y - y'}$ are Euclidean projections onto $\mc X$ and $\mc Y^N$, and the sampled gradients are given by}
	\[
	\ba{rcl}
	G_x(x(t),y(t);\xi_t) 
	& \!\!\! = \!\!\! & 
	\nabla_x\Psi( x(t),y(t) ;\xi_t)
	\\ 
	G_y(x(t),y(t);\xi_t) 
	& \!\!\! = \!\!\! & 
	\nabla_y\Psi( x(t),y(t) ;\xi_t)
	\ea
	\] 
In our multi-agent MDP setup, however, each agent receives samples $\xi_t$ from a Markov process whose state distribution at time $t$ is $P_t$, where $P_t$ converges to the unknown distribution $\Pi$ with a geometric rate. Thus, i.i.d.\ samples from the stationary distribution $\Pi$ are not available. \tc{black}{Since i.i.d.\ sampling-based convergence guarantees may not hold for correlated samples~\cite{gyorfi1996averaged}, it is important to examine the ergodic stochastic optimization scenario in which samples are taken from a stochastic process~\cite{duchi2012ergodic}; a recent application for the centralized GTD can be found in~\cite{wang2017finite}. In particular, we are interested} in designing and analyzing distributed algorithms for stochastic saddle point problem~\eqref{eq.main} in the ergodic setting. 

	\vspace*{-2ex}
\section{Main result}
\label{sec.alg}

We now present the main results of the paper: a fast algorithm for the multi-agent learning. We propose a distributed stochastic primal-dual algorithm in Section~\ref{sec.DHPD}, introduce underlying assumptions in Section~\ref{sec.assumptions}, and \tc{black}{establish a finite-time performance bound} in Section~\ref{sec.error}.

	\vspace*{-2ex}
\subsection{\tc{black}{Distributed homotopy primal-dual algorithm}}
	\label{sec.DHPD}

In this section, we extend stochastic primal-dual algorithm~\eqref{eq.mrsa} to the multi-agent learning setting. To solve the stochastic saddle point program~\eqref{eq.main} in a distributed manner, the algorithm maintains $2N$ primal-dual pairs of vectors $z_{j,k}(t) \DefinedAs (x_{j,k}(t),y_{j,k}(t))$ and $z'_{j,k}(t) \DefinedAs (x'_{j,k}(t),y'_{j,k}(t))$, which belong to $\mc X \times\mc Y$. \tc{black}{In the $k$th iteration round at time $t$,} the $j$th agent computes local gradient using the private objective $\Psi_j( z_{j,k}(t);\xi_{k,t})$,
\[
G_j(z_{j,k}(t);\xi_{k,t}) 
\, \DefinedAs \,
\left[
\ba{c}
G_{j,x} (z_{j,k}(t);\xi_{k,t}) 
\\[0.cm]
G_{j,y} (z_{j,k}(t);\xi_{k,t})
\ea
\right]
\] 
and receives the vectors $\{x_{i,k}'(t),i\in {\mc N}_j\}$ from its neighbors ${\mc N}_j$. Here, $G_{j,x} (z_{j,k}(t);\xi_{k,t})$ and $G_{j,y} (z_{j,k}(t);\xi_{k,t}) $ are gradients of $\Psi_j( z_{j,k}(t);\xi_{k,t})$ with respect to $x_{j,k}(t)$ and $y_{j,k}(t)$, respectively.

The primal iterate $x_{j,k}(t)$ is updated using a convex combination of the vectors $\{x_{i,k}'(t),i \in {\mc N}_j \}$ and the dual iterate $y_{j,k}(t)$ is modified using the mirror descent update. We model the convex combination as a mixing process over the graph $\mc G$ and assume that the mixing matrix $W$ is a doubly stochastic, 
\be
\non
\ba{rcl}
\displaystyle\sum_{i \, = \, 1}^N W_{ij} 
	& \!\!  = \!\!  & 
	\displaystyle\sum_{i \, \in \, {\mc N}_j} \! W_{ij} \, = \, 1, \text{ for all } j \in \mc V
	\\[0.2cm]
	\displaystyle\sum_{j \, = \, 1}^N W_{ij} 
	& \!\!  = \!\! & 
	\displaystyle\sum_{j \, \in \, {\mc N}_i} \! W_{ij} \, = \, 1, \text{ for all } i \in \mc V
\ea
\ee
where $W_{ij} >0$ for $(i,j)\in\mc E$. For a given learning rate $\eta_k$, each agent updates primal and dual variables according to Algorithm~\ref{alg.main}, where $\mathcal P_{\mc X}(\,\cdot\,)$ and $\mathcal P_{\mc Y}(\,\cdot\,)$ are Euclidean projections onto domains $\mc X$ and ${\mc Y}^{N}$.

The iteration counters $k$ and $t$ are used in our Distributed Homotopy Primal-Dual (DHPD) Algorithm, i.e., Algorithm~\ref{alg.main}. \tc{black}{The homotopy approach varies certain parameter for multiple rounds, where each round takes an estimated solution from the previous round as a starting point. We use the learning rate as a homotopy parameter in our algorithm. At the initial round $k=1$, problem~\eqref{eq.main} is solved with a large learning rate $\eta_1$ {and,} in subsequent iterations, the learning rate is gradually decreased until a desired error tolerance is reached. }

For a fixed learning rate $\eta_k$, we employ the distributed stochastic primal-dual method to solve~\eqref{eq.main} and obtain an approximate solution that is given by a time-running average of primal-dual pairs,
	\be
	\label{eq.avout}
	\hat{x}_{j,k} 
	\, \DefinedAs \, 
	\dfrac{1}{T_{k}} \, \sum_{t \, = \, 1}^{T_{k}} x_{j,k}(t),
	~
	\hat{y}_{j,k} 
	\, \DefinedAs \, 
	\dfrac{1}{T_{k}} \, \sum_{t \, = \, 1}^{T_{k}}y_{j,k}(t).
\ee
\tc{black}{These are used as initial points for the next learning rate $\eta_{k+1}$. At round $k$, each agent $j$ performs primal-dual updates with $T_k$ iterations, indexed by time $t$. At next round $k+1$, we initialize primal and dual iterations using the previous approximate solutions $\hat{x}_{j,k}$ and $\hat{y}_{j,k}$, reduce the learning rate by half, $\eta_{k+1}=\eta_k/2$, and double the number of the inner-loop iterations, $T_{k+1}=2T_k$. The number of inner iterations in the $k$th round is $T_k$ and the number of total rounds is $K$.}

\tc{black}{The homotopy approach not only provides outstanding practical performance but it also facilitates an effective iteration complexity analysis~\cite{xiao2013proximal}. In particular, for stochastic strongly convex programs, the rate faster than $O(1/\sqrt T)$ was established in~\cite{tsianos2012distributed,yang2015rsg}; other fast rate results can be found in~\cite{xu2016homotopy,wei2018solving}. To the best of our knowledge, we are the first to show that the homotopy approach can be used to solve distributed stochastic saddle-point problems with convergence rate better than $O(1/\sqrt T)$.}

\begin{algorithm}
	\caption{Distributed Homotopy Primal-Dual Algorithm}
	\label{alg.main}
	\textbf{Initialization:} $x_{j,1}(1) = x_{j,1}'(1) = 0$, $y_{j,1}(1) = y_{j,1}'(1) = 0$ for all $j\in\mathcal{V}$; $T_1$, $\eta_1$, $K$ 
	\\[0.2cm]
	\textbf{For} $k = 1$ to $K$ \textbf{do} 
	\\
	\vspace{-4mm}
	\begin{enumerate}
		\item 
		\textbf{For} $t = 1$ to $T_k-1$ \textbf{do} \Comment{All $j\in\mathcal{V}$}
		\\[-0.25cm]
		\begin{itemize}
			\item Primal update, \Comment{Distributed Averaging}
		\end{itemize}	
			\[
			\begin{array}{lll}
			x_{j,k}'(t+1) 
			& \!\!  = \!\!  & 
			\displaystyle\sum_{i \, = \, 1}^N W_{ij}\, x_{i,k}'(t) 
			\; - \; 
			\eta_k \,G_{j,x}(x_{j,k} (t),y_{j,k}(t);\xi_{k,t})
			\\[0.2cm]
			x_{j,k}(t+1) 
			& \!\! = \!\! & 
			\mc P_{\mc X}( \, x_{j,k}'(t+1) \, )
			\end{array}
			\]
			\begin{itemize}
			\item Dual update, \Comment{Local Update}
			\end{itemize}
			\vspace*{0.2cm}
			\[
			\begin{array}{lll}
			y_{j,k}'(t+1) 
			& \!\!  = \!\!  & 
			y_{j,k}'(t) \,+\, \eta_k \,G_{j,y}(x_{j,k}(t),y_{j,k}(t);\xi_{k,t})
			\\[0.1cm]
			y_{j,k}(t+1) 
			& \!\!  = \!\!  & 
			\mc P_{\mc Y}( \,y_{j,k}'(t+1)\, )
			\end{array}
			\]
		\textbf{end for}
		\vspace*{0.2cm}
		\item Restart initialization, \Comment{All $j\in\mathcal{V}$}
		\[
		\ba{rcl}
		\big( \,x_{j,k+1}(1),\, y_{j,k+1} (1)\,\big) & \!\!  = \!\!  &   (\,\hat{x}_{j,k},\, \hat{y}_{j,k}\,) \text{ (see~\eqref{eq.avout}) }
		\\[0.1cm]
		\big( \,x_{j,k+1}'(1),\, y_{j,k+1}'(1)\,\big) & \!\!  = \!\!  &  (\,x_{j,k+1}(1),\, y_{j,k+1}(1)\,)
		\ea
		\]
		\item Update stepsize,\,horizon: $\eta_{k+1} \,=\, \tfrac{1}{2}\,\eta_k,\,T_{k+1} \,=\, 2\,T_k$
	\end{enumerate}
	\vspace*{0.2cm}
	\textbf{end for}
	\\[0.2cm]
	\textbf{Output:} $(\hat{x}_{j,K}, \hat{y}_{j,K})$ for all $j\in\mathcal{V}$
\end{algorithm}

\tc{black}{In Section~\ref{sec.error}, we use the primal optimality gap, 
\begin{equation}\label{eq:objgap}
	\text{Err}(\hat{x}_{i,k}) 
	\, \DefinedAs \,
	\dfrac{1}{N} \sum_{j \, = \, 1}^{N} \left( f_j(\hat{x}_{i,k})- f_j(x^\star) \right)
\end{equation}
to quantify the distance of the running local average, $\hat{x}_{i,k}\DefinedAs (1 / T_{k})  \sum_{t = 1}^{T_k} x_{i,k}(t)$, for the $i$th agent from the optimal solution $x^\star$. The primal optimality gap measures performance of each agent in terms of MSPBE which is described by the global objective function~\eqref{eq.mspbe}, or equivalently by~\eqref{eq.mspbesum}. }

\begin{remark}[Multi-agent policy evaluation]
	{\color{black} For the multi-agent policy evaluation problem, we take $\psi_j (x, y_j) = y_j^T(Ax-  b_j)  - \tfrac{1}{2} \, y_j^TCy_j$. Since $\psi_j$ depends linearly on $A$, $b_j$, and $C$, by replacing expectations in~\eqref{eq.AbC} with the corresponding samples we obtain unbiased gradients,
		\[
		\ba{rcl}
		G_{j,x} (z_{j,k}(t);\xi_{k,t})  
		& \!\!\! = \!\!\! & 
		(\phi(s) -\gamma\phi(s') ) \phi(s)^T y_{j,k}(t)
		\\[0.1cm]
		G_{j,y} (z_{j,k}(t);\xi_{k,t}) 
		& \!\!\! = \!\!\! &  
		\phi(s) (\phi(s) -\gamma\phi(s') )^T x_{j,k}(t) 
		- 
		{R}_j^{\pi}(s)\phi(s) - \phi(s)\phi(s)^T y_{j,k}(t).
		\ea
		\]
		where $s$, $s'$ are two consecutive states that evolve according to the underlying Markov process $\xi_{k,t}$ indexed by time $t$ and the iteration round $k$. In each iteration, Algorithm~\ref{alg.main} requires $O(dN^2)$ operations where $d$ is the problem dimension (or the feature dimension in linear approximation) and $N$ is the total number of agents. For a single-agent problem, $O(d)$ operations are required which is consistent with GTD algorithm~\cite{sutton2009fast}.}
		\end{remark}

	\vspace*{-2ex}
\subsection{Assumptions}\label{sec.assumptions}

\tc{black}{We now formally state assumptions required to establish our main result in Theorem~\ref{thm.main} that quantifies finite-time performance of Algorithm~\ref{alg.main} for stochastic primal-dual optimization problem~\eqref{eq.main}.}

\begin{assumption}[Convex compact domain]\label{as:domain}
	The feasible sets $\mc X$ and $\mc Y$ contain the origin in $\mathbb{R}^d$ and they are convex and compact with radius $r> 0$, i.e., \mbox{$\sup_{x \, \in \, {\mathcal X}, \, y \, \in \, {\mc Y}}\|(x,y)\|^2\leq r^2$.} 
\end{assumption}  

\begin{assumption}[Convexity and concavity]\label{as:basic}
	The function $\psi_j(x,y_j)$ in~\eqref{eq.main} is convex in $x$ for any fixed $y_j\in\mc Y$, and is strongly concave in $y_j$ for any fixed $x\in \mc X$, i.e., for any $x,x'\in\mc X$ and $y_j,y_j'\in \mc Y$, there exists $L_y>0$ such that
	\[
	\begin{array}{rcl}
	\psi_j(x,y_j) & \!\! \geq \!\! & \psi_j(x',y_j) \; +  \; \dotp{\nabla_x\psi_j (x',y_j)}{x - x'}
	\\[0.1cm]
	\psi_j(x,y_j) & \!\! \leq \!\! & \psi_j(x,y_j') \; - \; \dotp{\nabla_y\psi_j(x,y_j')}{y_j - y_j'} \; - \; \dfrac{L_y}{2} \, \|y_j - y_j'\|^2.
	\end{array}
	\]
	Moreover, $f_j(x) \DefinedAs \max_{y_j \in \mc Y} \psi_j (x, y_j)$ is strongly convex, i.e., for any $x,x' \in\mc X$, there exists $L_x>0$ such that
	\[
	f_j(x) 
	\, \geq \,
	f_j(x') \,+\, \dotp{\nabla f_j(x')}{x-x'} \,+\, \dfrac{L_x}{2} \, \Vert x - x' \Vert^2.
	\]
\end{assumption}

\begin{assumption}[Bounded gradient]\label{as:var}
	For any $t$ and $k$, there is a positive constant $c$ such that the sample gradient $G_j(x,y_j;\xi_{k,t})$ satisfies,
	\[
			\Vert G_j(x,y_j;\xi_{k,t}) \Vert 
			\, \leq  \,
			c,
			~ 
			\mbox{for~all}~x \, \in \, {\mc X}, ~ y_j \, \in \, {\mc Y}
	\]
	with probability one.
\end{assumption}

	\begin{remark}
	\tc{black}{Jensen's inequality can be combined with Assumption~\ref{as:var} to show that the population gradient is also bounded, i.e., $\Vert g_j(x,y_j)\Vert \leq c$, for all $x\in\mc X$ and $y_j\in \mc Y$, where 
	\[
	g_j(x,y_j)
	\, \DefinedAs \,
	\left[
	\ba{c}
	g_{j,x}( x,y_j )
	\\[0.cm]
	g_{j,y}( x,y_j )
	\ea
	\right]
	\, = \,
	\left[
	\ba{c}
	\nabla_x\psi_j(x,y_j)
	\\[0.cm]
	\nabla_{y}\psi_j(x,y_j)
	\ea
	\right].
	\]}
\end{remark}
\begin{assumption}[Lipschitz gradient]\label{as:grad}
	For any $t$ and $k$, there exists a positive constant $L$ such that for any $x,x'\in\mc X$ and $y_j,y_j'\in\mc Y$, we have 
	\[
		\begin{array}{rcl}
			\Vert G_j(x,y_j;\xi_{k,t}) - G_j(x',y_j;\xi_{k,t}) \Vert &\!\!  \leq \!\! & L\, \Vert x - x' \Vert
			\\[0.1cm]
			\Vert G_j(x,y_j;\xi_{k,t}) - G_j(x,y_j';\xi_{k,t}) \Vert & \!\!  \leq \!\!  & L\, \Vert y_j - y_j' \Vert
		\end{array}
	\]
	with probability one.
\end{assumption}

We also recall some important concepts from probability theory. The total variation distance between distributions $P$ and $Q$ on a set 
$\Xi \subset \mathbb{R}^{|S|}$ is given by
	\beq
	\ba{rrl}
	d_{\text{tv}} (P,Q) 
	& \!\!  \DefinedAs \!\!  & 
	\displaystyle{\int_{\Xi}} \vert p(\xi) - q(\xi) \vert \, \mathrm{d} \mu(\xi) 
	\; = \;
	2 \displaystyle{\sup_{{E} \, \subset \, \Xi}}
	\,
	\vert P({E}) - Q({E}) \vert
	\ea
	\non
	\eeq
where distributions $P$ and $Q$ (with densities $p$ and $q$) are continuous with respect to the Lebesgue measure $\mu$, and the supremum is taken over all measurable subsets of $\Xi$. 

We use the notion of mixing time to evaluate the convergence speed of a sequence of probability measures generated by a Markovian process to its (unique) stationary distribution $\Pi$, whose density $\pi$ is assumed to exist. Let $\mc F_{k,t}$ be the $\sigma$-field generated by the first $t$ samples at round $k$, $\{ \xi_{k,1},\ldots,\xi_{k,t} \}$, drawn from $\{ P_{k,1},\ldots,P_{k,t} \}$, where $P_{k,t}$ is the probability measure of the Markovian process at time $t$ and round $k$. Let $P_{k,t}^{[s]}$ be the distribution of $\xi_{k,t}$ conditioned on $\mathcal{F}_{k,s}$ (i.e., given samples up to time slot $s$, $\{ \xi_{k,1},\ldots,\xi_{k,s} \}$) at round $k$, whose density $p_{k,t}^{[s]}$ also exists. The mixing time for a Markovian process is defined as follows~\cite{duchi2012ergodic}.

\begin{definition}\label{def.mixing}
	The total variation mixing time $\tau_{\normalfont\text{tv}}(P_k^{[s]},\varepsilon)$ of the Markovian process conditioned on the $\sigma$-field of the initial $s$ samples $\mathcal{F}_{k,s}=\sigma(\xi_{k,1},\ldots,\xi_{k,s})$ is the smallest positive integer $t$ such that $d_{\normalfont\text{tv}}(P_{k,s+t}^{[s]},\Pi) \leq \varepsilon$,
	\[
	\tau_{\normalfont\text{tv}}(P_k^{[s]},\varepsilon) 
	= 
	\inf
	\bigg\{ t-s \, \Big| \,  t\in\mathbb{N},\int_{\Xi} \vert p_{k,t}^{[s]}(\xi) - \pi(\xi) \vert \, \mathrm{d}\mu(\xi) \leq\varepsilon \bigg\}.
	\]
\end{definition}

The mixing time $\tau_{\normalfont\text{tv}}(P_k^{[s]},\varepsilon)$ measures the number of additional steps required until the distribution of $\xi_{k,t}$ is within $\varepsilon$ neighborhood of the stationary distribution $\Pi$ given the initial $s$ samples, $\{ \xi_{k,1},\ldots,\xi_{k,s} \}$.

\begin{assumption}\label{as.mc_general}
\tc{black}{The underlying Markov chain is irreducible and aperiodic, i.e., there exists $\Gamma> 0$ and $\rho\in(0,1)$ such that  $\mathbb{E}[\,d_{\normalfont\text{tv}} (P_{k,t+\tau}^{[t]},\Pi )\,]\leq \Gamma\rho^\tau$ for all $\tau\in\mathbb{N}$ and all $k$.}
\end{assumption}
\tc{black}{Furthermore, we have
\begin{equation}\label{eq:mixing-time}
	\tau_{\normalfont\text{tv}}(P_k^{[s]},\varepsilon) 
	\, \geq \,
	\l \lceil \dfrac{\log \, (\Gamma/\varepsilon)}{\lvert \log \, \rho \rvert}\r\rceil + 1, 
	~
	\mbox{for~all}~k, \, s \, \in \, {\mathbb N}
\end{equation}
where $\lceil\,\cdot\,\rceil$ is the ceiling function and $\varepsilon \leq \Gamma$ specifies the distance to the stationarity; also see~\cite[Theorem~4.9]{levin2017markov}. }

		\vspace*{-2ex}
\subsection{\color{black}Finite-time performance bound}
	\label{sec.error}

For stochastic saddle point problem~\eqref{eq.main}, we establish \tc{black}{a finite-time error bound} in terms of the average primal optimality gap in Theorem~\ref{thm.main} where the total number of iterations in Algorithm~\ref{alg.main} is given by $T\DefinedAs\sum_{k \, = \, 1}^{K}T_k = (2^K-1)T_1$.

\begin{theorem}\label{thm.main}
	{\color{black}Let Assumptions~\ref{as:domain}--\ref{as.mc_general} hold.} Then, for any $\eta_1\geq 1/(4/L_y+2/L_x)$ and any $T_1$ and $K$ that satisfy
	\begin{equation}
	\label{eq.kt}
	T_1
	\, \geq \,
	\tau 
	\, \DefinedAs \,
	\l\lceil\,\dfrac{\log \, (\Gamma T) }{|\log\rho|}\,\r\rceil 
	\, + \,
	1
	\end{equation}
	the output $\hat{x}_{j,K}$ of Algorithm~\ref{alg.main} provides the solution to problem~\eqref{eq.main} with the following upper bound  
	\be
	\dfrac{c(r L + c)}{T}
	\left(
	\dfrac{ C_1 \log^2 ( T\sqrt{N})}{1 - \sigma_2(W)} 
	\, + \,
	C_2 (1 + T_1)
	\right)
	\label{eq.bound}
	\ee
\tc{black}{on $(1/N) \sum_{j = 1}^{N} \mathbb{E}[\,\text{Err}(\hat{x}_{j,K})\,]$, where the primal optimality gap $\text{Err}(\hat{x}_{j,K})$ is defined in~\eqref{eq:objgap}, $r$ is the bound on feasible sets $\mc X$ and $\mc Y$ in Assumption~\ref{as:domain}, $c$ is the bound on sample gradients in Assumption~\ref{as:var},} $C_1$ and $C_2$ are constants independent of {$T$}, $\sigma_2(W)$ is the second largest eigenvalue of $W$, and $N$ is the total number of agents.
\end{theorem}
\begin{remark}[MSPBE minimization for multi-agent policy evaluation]
	\tc{black}{For $\psi_j (x, y_j) =  y_j^T ( Ax-  b_j)- \frac{1}{2}  y_j^T  C y_j $, Assumption~\ref{as.AC} guarantees that $A$ is full rank and that $C$ is positive definite. Since all features and rewards are bounded, Assumptions~\ref{as:basic}-\ref{as:grad} hold with
	\[
		\ba{rcl}
		L_x 
		& \!\!\! = \!\!\! & 
		\sigma_{\max}(A^T A)/\sigma_{\min}(C),
		~
		L_y \, = \, \sigma_{\min}(C)
	\\[0.1cm]
	c 
	& \!\!\! \geq \!\!\! & 
	\sqrt{(2\beta_1^2 + \beta_2^2) r^2 +\beta_0^2}
	\\[0.11cm]
	L 
	& \!\!\! \geq \!\!\! & 
 	\max(\sqrt{\beta_1^2+\beta_2^2},\beta_1)
	\ea
	\]
	where $\beta_0, \beta_1$ and $\beta_2$ provide upper bounds to
	$\Vert {R}_j^{\pi}(s)\phi(s)\Vert \leq \beta_0$, 
	$\Vert \phi(s) (\phi(s)-\gamma\phi(s'))^T \Vert \leq \beta_1$, and 
	$\Vert \phi(s) \phi(s)^T \Vert \leq \beta_2$. The unique minimizer of~\eqref{eq.mspbe} and the expression~\eqref{eq.fj-new} that results from Fenchel duality validate Assumption~\ref{as:domain} with
	\[
	r 
	\,\geq\, 
	\dfrac{2\beta_0}{\sigma_{\min}(C)}
	\left(
	\dfrac{\beta_1^2 \sigma_{\max}(C)}{\sigma_{\min}(A^T A)\sigma_{\min}(C)} 
	\, + \, 
	1
	\right).
	\]
	In practice, when some prior knowledge about the model is available, e.g., when generative models or simulators can be utilized, samples from a near stationary state distribution under a given policy can be used to estimate these parameters.  }
\end{remark}

\begin{remark}[Optimal performance bound and selection of parameters]
	\tc{black}{As long as \tc{black}{$T>\tau =\lceil\,{\log \, (\Gamma T) }/{|\log\rho|}\,\rceil$}, we can find $T_1$ and $K$ such that condition~\eqref{eq.kt} holds. In particular, choosing $T_1 = \tau$ and $K = \log(1+T/\tau)$ gives the desired $O(\log^2 (T \sqrt{N} )/T)$ scaling of finite-time performance bound~\eqref{eq.bound}. In general, to satisfy condition~\eqref{eq.kt}, we can choose $T_1$ and $K$ such that $T_1\geq \l\lceil\,({K + \log \, (\Gamma T_1) })/{|\log\rho|}\,\r\rceil +1$. Hence, performance bound~\eqref{eq.bound} scales as $O((K^2 + \log^2T_1)/T)$. Time-running average~\eqref{eq.avout} is used as the output of our algorithm and when the algorithm is terminated in an inner loop $K$, the time-running average from previous inner loop can be used as an output and our performance bound holds for $K-1$.}
\end{remark}

\begin{remark}[Mixing time]
	\tc{black}{When $\varepsilon = 1/T$, $\tau=\lceil\,{\log \, (\Gamma T) }/{|\log\rho|}\,\rceil $ provides a lower bound on the mixing time (cf.~\eqref{eq:mixing-time}), where $\Gamma> 0$ and $\rho\in(0,1)$ are given in Assumption~\ref{as.mc_general}. Thus, performance bound~\eqref{eq.bound} in Theorem \ref{thm.main} depends on how fast the process $P_k^{[s]}$ reaches $1/T$ mixing via $\tau$. In particular, when samples are independent and identically distributed, $0$-mixing is reached in one step, i.e., we have $\tau =1$ and $\epsilon = 0$. Hence, by setting $T_1=1$, performance bound~\eqref{eq.bound} simplifies to $O(\log^2 (T)/T)$.}
\end{remark}

\begin{remark}[Network size and topology]
	In \tc{black}{finite-time performance bound~\eqref{eq.bound},} the dependence on the network size $N$ and the spectral gap $1-\sigma_2(W)$ of the mixing matrix $W$ is quantified by ${\log^2( T \sqrt{N})}/{(1-\sigma_2(W))}$. We note that $W$ can be expressed using the Laplacian $L$ of the underlying graph, $W = I - D^{1/2}L D^{1/2}/(\delta_{\max}+1)$, where $D \DefinedAs \diag \, (\delta_1,\ldots,\delta_N)$, $\delta_i$ is the degree of node $i$, and $\delta_{\max} \DefinedAs \max_i \delta_i$. The algebraic connectivity of the network $\lambda_{N-1}(L)$, i.e., the second smallest eigenvalue of $L$, can be used to bound $\sigma_2(W)$. In particular, for a ring with $N$ nodes, we have $\sigma_2(W) = \Theta(1/N^2)$; for other topologies, see~\cite[Section~6]{duchi2012dual}.
\end{remark}

		\vspace*{-2ex}
\section{\tc{black}{Finite-time performance} analysis: proofs}
\label{sec.convergence}

In this section, we \tc{black}{study finite-time performance} of the distributed homotopy primal-dual algorithm described in Algorithm~\ref{alg.main}. We define auxiliary quantities in Section~\ref{sec.pre}, present useful lemmas in Section~\ref{sec.lem}, and provide the proof of Theorem~\ref{thm.main} in Section~\ref{sec.thm}.  

		\vspace*{-2ex}
\subsection{Setting up the analysis}
\label{sec.pre}

	\tc{black}{We first introduce three types of averages that are used to describe sequences generated by the primal update of~Algorithm~\ref{alg.main}.} The average value of $x_{j,k}(t)$ over all agents is denoted by $\bar{x}_{k}(t) \DefinedAs \frac{1}{N}\sum_{j = 1}^{N}x_{j,k}(t)$, the time-running average of $x_{j,k}(t)$ is $\hat{x}_{j,k} \DefinedAs \frac{1}{T_k}\sum_{t = 1}^{T_k}x_{j,k}(t)$, the averaged time-running average of $x_{j,k}(t)$ is given by $\tilde{x}_k \DefinedAs \frac{1}{N} \sum_{j = 1}^N \hat{x}_{j,k} = \frac{1}{T_k} \sum_{t = 1}^{T_k}\bar{x}_{k}(t)$, and two auxiliary averaged sequences are, respectively, given by
	$
\bar{x}_{k}'(t) \DefinedAs \frac{1}{N} \sum_{j = 1}^{N}x_{j,k}'(t)
	$ 
and \tc{black}{$\underbar{$x$}_{k}(t)  \DefinedAs  P_{\mc X}( \bar{x}_{k}'(t) )$.}
Since the mixing matrix $W$ is doubly stochastic, the primal update $\bar{x}_{k}'(t)$ has a simple `centralized' form,
\begin{equation}\label{eq:simple-update}
\overline x_{k}'(t+1) 
\,=\,
 \bar{x}_{k}'(t) -  \dfrac{\eta_k}{N}\displaystyle\sum_{j\,=\,1}^N G_{j,x}(z_{j,k}(t);\xi_{k,t}).
\end{equation}
\tc{black}{We now utilize the approach similar to the network averaging analysis in~\cite{duchi2012dual} to quantify how well the agent $j$ estimates the network average at round $k$.}

\begin{lemma}
	\label{lem.netavearaging}
	Let Assumption~\ref{as:var} hold, let $W$ be a doubly stochastic mixing matrix over graph $\mc G$, let $\sigma_2(W)$ denote its second largest singular value, and let a sequence $x_{j,k}(t)$ be generated by Algorithm~\ref{alg.main} for agent $j$ at round $k$. Then,
	\begin{equation}
	\non
		\dfrac{1}{T_k}\displaystyle\sum_{t\,=\,1}^{T_k} \mathbb{E}[\,\Vert x_{j,k}(t) -\bar{x}_{k}(t) \Vert\,]
		\,\leq\, 
		{\color{black}2\Delta_k}
	\end{equation}
	where $\Delta_k$ is given by
	\[
	\dfrac{2\eta_k {\color{black}c} \log(\sqrt{N} T_k)}{1-\sigma_2(W)}
	\; + \; \dfrac{4 {\color{black}c}}{T_k} \l(\dfrac{\log(\sqrt{N}T_k)}{1-\sigma_2(W)} \; + \; 1\r)\displaystyle\sum_{l\,=\,1}^{k}\eta_l T_{l}  \; + \; 2\eta_k {\color{black}c}
	\]
\end{lemma}
	\begin{proof}
		See Appendix~\ref{sec.proof.netavearaging}.
	\end{proof}	

The dual update $y_{j,k}'(t+1)$ behaves in a similar way as~\eqref{eq:simple-update}. We utilize the following classical online gradient descent to analyze the behavior of $\bar{x}_{k}'(t)$ and $y_{j,k}'(t)$ under projections $\mathcal P_{\mc X}(\,\cdot\,)$ and $\mathcal P_{\mc Y}(\,\cdot\,)$. 

\begin{lemma}[\!\!\cite{Zinkevich03ICML}]\label{lem:online-bound}
	Let $\mathcal{U}$ be a convex closed subset of $\mathbb{R}^d$, let $\{g(t)\}_{t=1}^{T}$ be an arbitrary sequence in $\mathbb{R}^d$, and let sequences $w(t)$ and $u(t)$ be generated by the projection, $w(t+1) = w(t) - \eta g(t)$ and $u(t+1) = \mathcal{P}_{\mc U}(\,w(t+1) \,)$, 
	where $u(1)\in \mc U$ is the initial point and $\eta>0$ is the learning rate. Then, for any fixed $u^\star\in\mathcal{U}$, 
	\[
	\sum_{t\,=\,1}^{T}\dotp{g(t)}{u(t) - u^\star} 
	\,\leq\,
	 \frac{\Vert u(1) - u^\star\Vert^2}{2\eta} \,+\, \frac{\eta}{2}\sum_{t\,=\,1}^{T}\Vert g(t)\Vert^2.
	\]
\end{lemma}
\begin{proof}
	See Appendix B.3 in~\cite{Zinkevich03ICML}.
\end{proof}

If $\hat{y}_{j,k}^\star \DefinedAs \text{argmax}_{y_j\,\in\,\mc Y}\, \psi_j(\hat{x}_{i,k}, y_j)$, using Fenchel dual~\eqref{eq.fj-new}, we have $f_j(\hat{x}_{i,k}) = \psi_j(\hat{x}_{i,k},\hat{y}_{j,k}^\star)$ and $f_j(x^\star)=\psi_j(x^\star, y_{j}^\star)$. \tc{black}{This allows us to express primal optimality gap~\eqref{eq:objgap} in terms of a primal-dual objective,}
\begin{equation}\label{eq:objgap1}
\text{Err}(\hat{x}_{i,k}) 
\,=\,
\frac{1}{N}\sum_{j\,=\,1}^{N} \left(\,\psi_j(\hat{x}_{i,k}, y_{j,k}^\star) -\psi_j(x^\star, y_j^\star)\,\right).
\end{equation}
Let $\hat{x}_k^\star \DefinedAs \text{argmin}_{x\,\in\,\mc X} \frac{1}{N}\sum_{j\,=\,1}^{N} \psi_j(x, \hat{y}_{j,k})$.  
To analyze optimality gap~\eqref{eq:objgap1}, we introduce a surrogate gap,
\be
\label{eq.surrogate}
\text{Err}'(\hat{x}_{i,k},\hat{y}_k)
\,\DefinedAs\,
\dfrac{1}{N}\!\displaystyle\sum_{j\,=\,1}^{N}\! \left(\psi_j(\hat{x}_{i,k},\hat{y}_{j,k}^\star)
-
\psi_j(x^\star,\hat{y}_{j,k}) \right)
\ee
as well as average primal optimality~\eqref{eq:objgap} and surrogate~\eqref{eq.surrogate} gaps, 
\[
	\ba{rrl}
\overline{\text{Err}}_k
 & \!\!\! \DefinedAs \!\!\! &
 \dfrac1N \displaystyle \sum_{j\,=\,1}^{N}\mathbb{E}[\,\text{Err}(\hat{x}_{j,k})\,],
 	\\
\overline{\text{Err}}_k' 
	& \!\!\! \DefinedAs \!\!\! &
\dfrac{1}{N} \displaystyle \sum_{j\,=\,1}^{N}\mathbb{E}[\,\text{Err}'(\hat{x}_{j,k},\hat{y}_k)\,].
	\ea
\]

In Lemma~\ref{lem:objective2PD}, we establish relation between the surrogate gap $\text{Err}'(\hat{x}_{i,k},\hat{y}_k)$ and the primal optimality gap $\text{Err}(\hat{x}_{i,k})$.
\begin{lemma}\label{lem:objective2PD}
	Let $\hat{x}_{i,k}$ and $\hat{y}_k$ be generated by Algorithm~\ref{alg.main} for agent  $i$ at round $k$.
	Then, $0 \leq \text{Err}(\hat{x}_{i,k})\leq  \text{Err}'(\hat{x}_{i,k},\hat{y}_k)$ and $0\leq \overline{\text{Err}}_k\leq\overline{\text{Err}}_k'$.
\end{lemma}
	\begin{proof}
		See Appendix~\ref{sec.proof.objective2PD}.
	\end{proof}

\begin{lemma}\label{lem:convexconcave}
	Let Assumption~\ref{as:basic} hold and let $\hat{x}_{i,k}$ and $\hat{y}_k$ be generated by Algorithm~\ref{alg.main} for agent  $i$ at round $k$. Then,
	\begin{subequations}
	\be
	\label{eq.cc1}
	\text{Err}(\hat{x}_{i,k}) 
	\,\geq \,
	\dfrac{L_x}{2}\, \Vert x^\star-\hat{x}_{i,k} \Vert^2
	\ee
	\vspace*{-0.35cm}
	\be
	\label{eq.cc2}
	\text{Err}'(\hat{x}_{i,k},\hat{y}_k) 
	\,\geq\,
	  \dfrac{L_y}{2N}\, \displaystyle\sum_{j\,=\,1}^{N}\Vert y_j^\star-\hat{y}_{j,k} \Vert^2
	\ee
	\vspace*{-0.2cm}
	\be
	\label{eq.cc3}
	\text{Err}'(\hat{x}_{i,k},\hat{y}_k) 
	\, \geq \,
	 \dfrac{L_y}{2N}\, \displaystyle\sum_{j\,=\,1}^{N}\Vert y_j^\star-\hat{y}_{j,k}^\star \Vert^2.
	\ee
	\end{subequations}
\end{lemma}
\begin{proof}
			 See Appendix~\ref{sec.proof.convexconcave}.
	\end{proof}

{\color{black}We are now ready to provide an overview of our remaining analysis. In Lemma~\ref{lem:decomposegap}, we use a sum of network errors (NET) and local primal-dual gaps (PDG) to bound surrogate gap~\eqref{eq.surrogate}. In Lemma~\ref{lem:boundII}, we provide a bound on local primal-dual gaps by a sum of local dual gaps and a term that depends on mixing time. We combine Lemma~\ref{lem:decomposegap} and Lemma~\ref{lem:boundII} and apply the restarting strategy to get a recursion on the surrogate gap. Finally, in Section~\ref{sec.thm}, we complete the proof by utilizing induction on the round $k$.}

	\vspace*{-2ex}
\subsection{Useful lemmas}
\label{sec.lem}

\tc{black}{We utilize convexity and concavity of $\psi_j$ with respect to primal and dual variables to decompose surrogate gap~\eqref{eq.surrogate} into parts that quantify the influence of network errors (NET) and local primal-dual gaps (PDG), respectively.}

\begin{lemma}\label{lem:decomposegap}
	Let Assumptions~\ref{as:basic} and~\ref{as:var} hold and let $\hat{x}_{i,k}$ and $\hat{y}_k$ be generated by Algorithm~\ref{alg.main} for agent $i$ at round $k$. Then,
	\[
	\text{Err}'(\hat{x}_{i,k},\hat{y}_k) 
	\,\leq\,
	 \normalfont\text{NET} \,+\, \normalfont\text{PDG}
	\]
	where 
	\[
	\ba{rcl}
	\!\!\normalfont\text{NET} 
	&\!\!\! =\!\!\! & 
	\dfrac{{\color{black}c}}{T_k}\!\displaystyle\sum_{t\,=\,1}^{T_k} \!\bigg(\!\Vert x_{i,k}(t) \!-\! \bar{x}_k(t) \Vert
	\!+\!
	\dfrac{1}{N}\displaystyle\sum_{j\,=\,1}^{N}\left\Vert  x_{j,k}(t) \!-\! \bar{x}_k(t)\right\Vert\!\bigg)
	\\[0.2cm]
	 \!\!\normalfont\text{PDG} 
	 &\!\!\! =\!\!\!  &
	 \dfrac{1}{NT_k}\!\displaystyle\sum_{t\,=\,1}^{T_k}\displaystyle\sum_{j\,=\,1}^{N} \big(\psi_j(x_{j,k}(t),  \hat{y}_{j,k}^\star) - \psi_j(x^\star, y_{j,k}(t))\big).
	\ea
	\]
	\begin{proof}	
		Applying the mean value theorem and boundedness of the gradient of $\psi_j(x,\hat{y}_{j,k}^\star)$ with respect to $x$, we have 
		\[
		\ba{rcl}
				\psi_j(\hat{x}_{i,k},\hat{y}_{j,k}^\star) \,-\, \psi_j(\tilde{x}_k,\hat{y}_{j,k}^\star) & \!\! \leq \!\! & {\color{black}c}\,\Vert \hat{x}_{i,k} - \tilde{x}_k \Vert
	\; \leq \;
				\dfrac{{\color{black}c}}{T_k}\,\displaystyle\sum_{t\,=\,1}^{T_k} \Vert x_{i,k}(t) - \bar{x}_k(t) \Vert
		\ea
		\]
		where the second inequality follows from the Jensen's inequality. Then, 
		breaking $\text{Err}'(\hat{x}_{i,k},\hat{y}_k)$ in~\eqref{eq.surrogate} by adding and subtracting $\tfrac{1}{N}\sum_{j = 1}^{N} \psi_j(\tilde{x}_k,\hat{y}_{j,k}^\star)$, we bound $\text{Err}'(\hat{x}_{i,k},\hat{y}_k)$ by 
		\[
		\dfrac{{\color{black}c}}{T_k}\!\displaystyle\sum_{t\,=\,1}^{T_k} \Vert x_{i,k}(t) - \bar{x}_k(t) \Vert
		+  \dfrac{1}{N}\!\displaystyle\sum_{j\,=\,1}^{N}\! ( \psi_j(\tilde{x}_k,\hat{y}_{j,k}^\star) - \psi_j(x^\star,\hat{y}_{j,k}) ).
		\]
		Next, we find a simple bound for the second sum. We recall that 
		$\tilde{x}_k \DefinedAs \frac{1}{T_k} \sum_{t = 1}^{T_k}\bar{x}_{k}(t)$ and $\hat{y}_{j,k} \DefinedAs \frac{1}{T_{k}}  \sum_{t = 1}^{T_{k}}y_{j,k}(t)$ and apply the Jensen's inequality twice to obtain,
		\be
		\label{eq.furtherub}
		\dfrac{1}{N T_k}\displaystyle\sum_{t\,=\,1}^{T_k} \sum_{j\,=\,1}^N\big(\psi_j(\bar{x}_k(t), \hat{y}_{j,k}^\star) - \psi_j(x^\star, y_{j,k}(t))\big).
		\ee
		Similarly, we have $\psi_j(\bar{x}_k(t),  \hat{y}_{j,k}^\star) 
		- \psi_j(x_{j,k}(t),  \hat{y}_{j,k}^\star)\leq {\color{black}c}\left\Vert \bar{x}_k(t) - x_{j,k}(t)\right\Vert $. Breaking~\eqref{eq.furtherub} by adding and subtracting $\frac{1}{NT_k}\sum_{t = 1}^{T_k}\sum_{j = 1}^{N}\psi_j(x_{j,k}(t),  \hat{y}_{j,k}^\star)$, we further bound~\eqref{eq.furtherub} by
		\[
		\ba{rcl}
		\dfrac{{\color{black}c}}{N T_k}\displaystyle\sum_{t\,=\,1}^{T_k} \sum_{j\,=\,1}^{N}\left\Vert x_{j,k}(t) - \bar{x}_k(t)\right\Vert 
		\,+\, 
		\normalfont\text{PDG}.
		\ea
		\]
		The proof is completed by combining all above bounds.
	\end{proof}
\end{lemma}

Lemma~\ref{lem:decomposegap} establishes a bound for the surrogate gap $\text{Err}'(\hat{x}_{i,k},\hat{y}_k)$ for agent $i$ at round $k$. The terms in \text{NET} describe the accumulated network error that measures the deviation of each agent's estimate from the average. On the other hand, \text{PDG} determines the average of primal-dual gaps incurred by local agents that are commonly used in the analysis of primal-dual algorithms~\cite{nemirovski2009robust,nedic2009subgradient}.

Next, we utilize the Markov mixing property to control the average of local primal-dual gaps \text{PDG}. First, we break the difference $\psi_j(x_{j,k}(t),  \hat{y}_{j,k}^\star) - \psi_j(x^\star, {y}_{j,k}(t))$ into a sum of $\psi_j(x_{j,k}(t),  \hat{y}_{j,k}^\star) - \psi_j(x_{j,k}(t), y_{j,k}(t))$ and $\psi_j(x_{j,k}(t), y_{j,k}(t)) - \psi_j(x^\star, {y}_{j,k}(t))$.
We now can utilize convexity and concavity of $\psi_j(x,y_j)$ to deal with these terms separately. Dividing the sum indexed by $t$ into two intervals, $1\leq t\leq T_k-\tau$ and $T_k-\tau+1\leq t\leq T_k$, the $\text{PDG}$ term can be bounded by $\text{PDG}\leq\text{PDG}^++\text{PDG}^-$, where
\be
\non
\ba{rcl}
\normalfont\text{PDG}^+ 
& \!\! = \!\! &
 \dfrac{1}{NT_k}\!\displaystyle\sum_{t\,=\,1}^{T_k-\tau}\displaystyle\sum_{j\,=\,1}^{N} \big( \dotp{g_{j,y}(z_{j,k}(t))}{\hat{y}_{j,k}^\star - y_{j,k}(t)} 
 	+  
\dotp{g_{j,x}(z_{j,k}(t))}{x_{j,k}(t) - x^\star}\!\big)
\\[0.2cm]
\normalfont\text{PDG}^- 
	& \!\! = \!\! &
 \dfrac{1}{NT_k}\!\displaystyle\sum_{t\,=\,T_k-\tau+1}^{T_k}\displaystyle\sum_{j\,=\,1}^{N}  \big( \dotp{g_{j,y}(z_{j,k}(t))}{\hat{y}_{j,k}^\star - y_{j,k}(t)} 
	+ 
\dotp{g_{j,x}(z_{j,k}(t))}{x_{j,k}(t) - x^\star}\!\big).
\ea
\ee
Here, $\tau$ is the mixing time of the ergodic sequence $\xi_{k,1},\ldots,\xi_{k,t}$ at round $k$. The intuition behind this is that, given the initial $t-\tau$ samples $\xi_{k,1},\ldots,\xi_{k,t-\tau}$, the sample $\xi_{k,t}$ is almost a sample that arises from the stationary distribution $\Pi$. With this in mind, we next show that an appropriate breakdown of the term $\text{PDG}^+$ enables applications of the martingale concentration~from Lemma~\ref{lem:martingale} and mixing time property~\eqref{eq:mixing-time}, thereby producing a gradient-free bound on primal-dual gaps.

\begin{lemma}\label{lem:boundII}
	Let Assumptions~\ref{as:domain}--\ref{as:grad} hold. For $T_k$ that satisfies $T_k\geq1+\lceil\,\log(\Gamma T_k)/|\log\rho|\,\rceil =\tau$, we have
	\[
	\mathbb{E}[\normalfont\text{PDG}]
	\, \leq \,
	\dfrac{{\color{black}c}}{N} \l( \dfrac{ {\color{black}8} }{\sqrt{T_k-\tau}} 
	\, + \,
	\dfrac{1}{ T_k} \r) \displaystyle\sum_{j\,=\,1}^{N} \expect{\,\l\Vert \hat{y}_{j,k}^\star 
	\, - \, 
	y_{j}^\star \r\Vert\,} 
	\, + \,
	\mathbb{E}[\,\normalfont\text{MIX}\,]
	\]
	where 
	\[
	\non
	\ba{rcl}
	\normalfont\text{MIX} 
	& \!\!\!=\!\!\! & 
	\dfrac{2{\color{black}r}L\,+\,\sqrt{2}{\color{black}c}}{NT_k} \!\displaystyle\sum_{t\,=\,\tau+1}^{T_k}\displaystyle\sum_{j\,=\,1}^{N} \left\Vert z_{j,k}(t-\tau)-z_{j,k}(t) \right\Vert
	~ +
	\\[0.2cm]
	&  & 
	\!\!\!\!\!\!\!\!\!\!\!\!\!\!\!\!\!\!\!\!\!\!\!\!\!\!
	\dfrac{1}{2\eta_kT_k} \bigg(\!\!\l\Vert x^\star \!-\! {\color{black}\underbar{$x$}_k(\tau\!+\!1)}\r\Vert^2 \,+\, \dfrac{1}{N}\!\displaystyle\sum_{j\,=\,1}^{N}\! \l\Vert \hat{y}_{j,k}^\star \!-\!y_{j,k}(\tau\!+\!1)  \r\Vert^2\!\!\bigg)
	~ +
	\\[0.2cm]
	&  & 
	\!\!\!\!\!\!\!\!\!\!\!\!\!\!\!\!\!\!\!\!\!\!\!\!\!\!
	\dfrac{{\color{black}c}}{NT_k}\!\displaystyle\sum_{t\,=\,\tau+1}^{T_k}\displaystyle\sum_{j\,=\,1}^{N} \l\Vert x_{j,k}(t) - {\color{black}\underbar{$x$}_k(t)}  \r\Vert
	+
	 \dfrac{2{\color{black}r c}(\tau +1) }{T_k} 
	+
	\eta_k {\color{black}c}^2.
	\ea
	\]
\end{lemma}
\begin{proof}
	Using Assumption~\ref{as:var}, $\text{PDG}^-$ can be upper bounded by
	\[
	\dfrac{{\color{black}c}}{NT_k}\displaystyle\sum_{t\,=\,T_k-\tau+1}^{T_k}\displaystyle\sum_{j\,=\,1}^{N}  \big(\left\Vert \hat{y}_{j,k}^\star-y_{j,k}(t) \right\Vert + \left\Vert x_{j,k}(t) -{x}^\star\right\Vert\big).
	\]
	Since the domain is bounded, this term is upper bounded by ${ 2 {\color{black}r c} \tau}/{T_k}$.
	
	Next, we deal with $\text{PDG}^+$. We divide each 
	$\langle g_{j,y}(z_{j,k}(t)), \hat{y}_{j,k}^\star -y_{j,k}(t)\rangle +\langle g_{j,x}(z_{j,k}(t)),x_{j,k}(t) - x^\star\rangle$
	 into a sum of five terms~\eqref{eq.md1}--\eqref{eq.md5} by adding and subtracting $G_{j,x}(z_{j,k}(t);\xi_{k,t+\tau})$ and $G_{j,y}(z_{j,k}(t);\xi_{k,t+\tau})$ into the first arguments of two inner products, respectively, and then inserting $y_j^\star$ into the second argument for the first resulting inner product,
	 \begin{subequations}
	 	\label{eq.md}
	\be
	\label{eq.md1}
	\dotp{g_{j,y}(z_{j,k}(t)) - G_{j,y}(z_{j,k}(t);\xi_{k,t+\tau})}{\hat{y}_{j,k}^\star -y_{j}^\star }
	\ee
	\be
	\label{eq.md2}
	 \dotp{g_{j,y}(z_{j,k}(t))-G_{j,y}(z_{j,k}(t);\xi_{k,t+\tau})}{ y_{j}^\star -y_{j,k}(t)}
	\ee
	\be
	\label{eq.md3}
	 \dotp{ G_{j,y}(z_{j,k}(t);\xi_{k,t+\tau})}{ \hat{y}_{j,k}^\star -y_{j,k}(t)} 
	\ee
	\be
	\label{eq.md4}
	 \dotp{g_{j,x}(z_{j,k}(t))-G_{j,x}(z_{j,k}(t);\xi_{k,t+\tau}) }{x_{j,k}(t)-x^\star}
	\ee
	\be
	\label{eq.md5}
	 \dotp{G_{j,x}(z_{j,k}(t);\xi_{k,t+\tau}) }{x_{j,k}(t) -x^\star}.
	\ee
	\end{subequations}
	We sum each of~\eqref{eq.md1}--\eqref{eq.md5} over $t=1,\ldots,T_k-\tau$ and $j=1,\ldots,N$, divide it by $NT_k$, and represent each of them using $S_1$ to $S_5$. Thus, $\mathbb{E}[\,\text{PDG}^+\,]= \mathbb{E}[\,S_1+S_2+S_3+S_4+S_5\,]$. We next bound each term separately.
	
	\textit{Bounding the term $\expect{\,S_1\,}$}: For agent $j$, we introduce a martingale difference sequence $\{X_j(t)\}_{t = 1}^{T_k}$,
	\[
	X_j(t) 
	\,=\, 
	g_{j,y}(z_{j,k}(t)) \,-\, G_{j,y}(z_{j,k}(t);\xi_{k,t+\tau}) \,-\, {\color{black}E_{j,k}(t)}
	\]
	where ${\color{black}E_{j,k}(t)} \DefinedAs \mathbb{E} [ \,g_{j,y}(z_{j,k}(t)) - G_{j,y}(z_{j,k}(t);\xi_{k,t+\tau}) \,\vert\, \mc F_{k,t}\,]$ and $M = {\color{black}4{\color{black}c}}$ in Lemma~\ref{lem:martingale}. This allows us to rewrite $S_1$ as
	\be
	\non
	S_1 
	 \; = \;
	  \dfrac{1}{N}\displaystyle\sum_{j\,=\,1}^{N} \dotp{\dfrac{1}{T_k}\sum_{t\,=\,1}^{T_k -\tau} X_j(t) }{\hat{y}_{j,k}^\star - y_{j}^\star } 
	 \; + \;
	  \dfrac{1}{N T_k}\displaystyle\sum_{j\,=\,1}^{N} \displaystyle\sum_{t\,=\,1}^{T_k-\tau}\dotp{ {\color{black} E_{j,k}(t)} } {\hat{y}_{j,k}^\star - y_{j}^\star}.
	\ee
	Since $(T_k-\tau)/ T_k\leq 1$, Lemma~\ref{lem:martingale} implies
	\begin{equation*}
	\expect{\,\left\Vert \dfrac{1}{T_k}\displaystyle\sum_{t\,=\,1}^{T_k-\tau} X_j(t) \right\Vert^2 \,} 
	\,\leq\,
	{\color{black}\dfrac{64 {\color{black}c}^2}{T_k -\tau}}.
	\end{equation*}
	Using Assumption~\ref{as:var} and the mixing time property~\eqref{eq:mixing-time}, we can bound {\color{black}$\norm{E_{j,k}(t)}$} by
	\be
	\label{eq.emixing}
	\ba{rcl}
	\norm{E_{j,k}(t)}
	& \!\! =\!\!  & 
	\l\Vert\displaystyle\int \!\! G_{i,y}(z_{i,k}(t);\xi) \big(\pi(\xi) - p_{k,t+\tau}^{[t]}(\xi)\big)\, d \mu(\xi) \r\Vert 
	\\[0.2cm]
	& \!\! \leq\!\!  & 
	{\color{black}c} \displaystyle\int \left\vert \pi(\xi) - p_{k,t+\tau}^{[t]}(\xi)\right\vert d \mu(\xi) 
	\; = \;
	{\color{black}c} \,d_{\normalfont\text{tv}}( P_{k,t+\tau}^{[t]},\Pi).
	\ea
	\ee
	Applying the triangle and Cauchy-Schwartz inequalities to $\expect{S_1}$ and using~\eqref{eq.emixing} lead to,
	\be
	\non
	\ba{rcl}
    \expect{ S_1 } 
    &\!\! \leq \!\!  & 
    \dfrac{1}{N}\, \dfrac{ {\color{black}8{\color{black}c}} }{\sqrt{T_k\,-\,\tau}} \displaystyle\sum_{j\,=\,1}^{N} \l\Vert \hat{y}_{j,k}^\star - y_{j}^\star \r\Vert 
    \; + \; \dfrac{{\color{black}c}}{N T_k}\displaystyle\sum_{j\,=\,1}^{N} \displaystyle\sum_{t\,=\,1}^{T_k-\tau} \expect{\,d_{\normalfont\text{tv}}(P_{k,t+\tau}^{[t]},\Pi)\,} \l\Vert \hat{y}_{j,k}^\star - y_{j}^\star \r\Vert
	\\[0.2cm]
	& \!\! \leq\!\!  &  
	 \dfrac{{\color{black}c}}{N} \left(\dfrac{ {\color{black}8}}{\sqrt{T_k\,-\,\tau}} + \dfrac{1}{T_k}\right)\displaystyle\sum_{j\,=\,1}^{N} \l\Vert \hat{y}_{j,k}^\star - y_{j}^\star \r\Vert. 
	\ea
	\ee
	where the last inequality follows from the mixing time property: if we choose $T_k$ such that $\tau= 1+\lceil\,\log(\Gamma T_k)/\log|\rho|\,\rceil\geq\tau_{\normalfont\text{tv}}(P_k^{[t]},{1}/{T_k})$, then we have 
	$\mathbb{E}[\,d_{\normalfont\text{tv}}( P_{k,t+\tau}^{[t]},\Pi)\,] \leq {1}/{T_k}$.
	
	\textit{Bounding the terms $\expect{S_2}$ and $\expect{S_4}$}: Using the Cauchy-Schwartz inequality and~\eqref{eq.emixing}, we can bound $\expect{ S_2 }$ by
	\be
	\non
	\ba{rcl}
	\expect{ S_2 } 
	& \!\! \leq\!\!  & 
	\dfrac{{\color{black}c}}{NT_k} \! \displaystyle\sum_{t\,=\,1}^{T_k-\tau}\displaystyle\sum_{j\,=\,1}^{N} \mathbb{E}[\,d_{\normalfont\text{tv}}(\Pi, P_{k,t+\tau}^{[t]})\, ] \,\Vert y_{j}^\star - y_{j,k}(t) \Vert
	\\[0.2cm]
	& \!\! \leq\!\!  & 
	\dfrac{{\color{black}c}}{N T_k^2}  \displaystyle\sum_{t\,=\,1}^{T_k-\tau}\displaystyle\sum_{j\,=\,1}^{N} \mathbb{E}[\,\Vert y_{j}^\star - y_{j,k}(t) \Vert\,]
	\\[0.2cm]
	& \!\! \leq \!\! & 
	\dfrac{{\color{black}r c} }{T_k}.
	\ea
	\ee
	Similarly, we have 
	\[
	\expect{S_4}  
	\,\leq\,
	 \dfrac{{\color{black}c}}{N T_k^2}\displaystyle\sum_{t\,=\,1}^{T_k-\tau} \displaystyle\sum_{j\,=\,1}^{N} \mathbb{E}[\,\Vert x_{j,k}(t)-x^\star \Vert\,] 
	 \,\leq\,
	  \dfrac{ {\color{black}r c}}{T_k}.
	\]
	
	\textit{Bounding the terms~$\mathbb{E}[ S_3 ]$ and~$\mathbb{E}[ S_5 ]$}: We re-index the sum in $S_3$ over $t$ and write it as
	\[ 
	\dfrac{1}{NT_k}\displaystyle\sum_{t\,=\,\tau+1}^{T_k}\displaystyle\sum_{j\,=\,1}^{N} \dotp{ G_{j,y}(z_{j,k}(t-\tau);\xi_{k,t})}{ \hat{y}_{j,k}^\star -y_{j,k}(t-\tau)}
	\]
	where each $\langle G_{j,y}(z_{j,k}(t-\tau);\xi_{k,t}), \hat{y}_{j,k}^\star -y_{j,k}(t-\tau)\rangle $ can be split into a sum of the following three inner products,
	\be
	\non
	\ba{c}
	\!\!\langle G_{j,y}(z_{j,k}(t\!-\!\tau);\xi_{k,t}) - G_{j,y}(z_{j,k}(t);\xi_{k,t}),\, \hat{y}_{j,k}^\star - y_{j,k}(t \!-\!\tau)\rangle
	\\[0.2cm]
	\langle G_{j,y}(z_{j,k}(t);\xi_{k,t}), \hat{y}_{j,k}^\star -y_{j,k}(t) \rangle 
	\\[0.2cm]
	\langle G_{j,y}(z_{j,k}(t);\xi_{k,t}), y_{j,k}(t) -y_{j,k}(t-\tau) \rangle.
	\ea
	\ee
	 From the Lipschitz continuity of the gradient, we have $\Vert G_{j,y}(z_{j,k}(t-\tau);\xi_{k,t}) -  G_{j,y}(z_{j,k}(t);\xi_{k,t}) \Vert \leq L\Vert z_{j,k}(t-\tau)-z_{j,k}(t) \Vert$. Combining this with the domain/gradient boundedness yields,
	\be
	\non
	\ba{rcl}
	S_3
	&\!\! \leq\!\! & \dfrac{{\color{black}r} L}{NT_k}\displaystyle\sum_{t\,=\,\tau+1}^{T_k}\displaystyle\sum_{j\,=\,1}^{N} \left\Vert z_{j,k}(t-\tau) - z_{j,k}(t) \right\Vert 
	\\[0.2cm]
	&&\!\!\!\!\!\!\!\! +\, \dfrac{1}{NT_k}\displaystyle\sum_{t\,=\,\tau+1}^{T_k}\displaystyle\sum_{j\,=\,1}^{N} \!\dotp{ G_{j,y}(z_{j,k}(t);\xi_{k,t}) }{ \hat{y}_{j,k}^\star -y_{j,k}(t)} 
	\\[0.2cm]
	&&\!\!\!\!\!\!\!\! +\, \dfrac{{\color{black}c}}{NT_k}\displaystyle\sum_{t\,=\,\tau+1}^{T_k}\displaystyle\sum_{j\,=\,1}^{N} \left\Vert y_{j,k}(t)- y_{j,k}(t-\tau) \right\Vert.
	\ea
	\ee
	Similarly, we have 
	\be
	\non
	\ba{rcl}
	S_5
	& \!\!  \leq \!\! & \dfrac{{\color{black}r} L}{NT_k}\displaystyle\sum_{t\,=\,\tau+1}^{T_k}\displaystyle\sum_{j\,=\,1}^{N} \left\Vert z_{j,k}(t-\tau)-z_{j,k}(t) \right\Vert 
	\\[0.2cm]
	&&\!\!\!\!\!\!\!\! +\, \dfrac{ {\color{black}c}}{NT_k}\displaystyle\sum_{t\,=\,\tau+1}^{T_k}\displaystyle\sum_{j\,=\,1}^{N} \left\Vert x_{j,k}(t-\tau)-x_{j,k}(t) \right\Vert
	\\[0.2cm]
	&&\!\!\!\!\!\!\!\! +\, \dfrac{1}{NT_k}\,\displaystyle\sum_{t\,=\,\tau+1}^{T_k}\displaystyle\sum_{j\,=\,1}^{N} \,\dotp{ G_{j,x}(z_{j,k}(t);\xi_{k,t}) }{ x_{j,k}(t) -x^\star}.
	\ea
	\ee
	Inserting {\color{black}$\underbar{$x$}_k(t)$} into the second argument of the inner product in the above bound on $S_5$ and using inequality $(\Vert a \Vert + \Vert b \Vert)^2 \leq 2 \Vert a \Vert^2+2\Vert b \Vert^2$, we bound $S_3+S_5$ by 
	\[
	\ba{rcl}
	 \!\! S_3\,+\,S_5
	 & \!\! \leq \!\! & \dfrac{2{\color{black}r}L+\sqrt{2}{\color{black}c}}{NT_k}\!\!\!\displaystyle\sum_{t\,=\,\tau+1}^{T_k}\displaystyle\sum_{j\,=\,1}^{N} \left\Vert z_{j,k}(t-\tau)-z_{j,k}(t) \right\Vert
	\\[0.2cm]
	&  &\!\!\!\!\!\!\!\!\!\!\!\!\!\!\!\!\!\!\!\!\!\!\!\! +\, \dfrac{1}{NT_k}\displaystyle\sum_{t\,=\,\tau+1}^{T_k}\displaystyle\sum_{j\,=\,1}^{N}\dotp{ G_{j,y}(z_{j,k}(t);\xi_{k,t}) }{ \hat{y}_{j,k}^\star - y_{j,k}(t)} 
	\\[0.2cm]
	&  &\!\!\!\!\!\!\!\!\!\!\!\!\!\!\!\!\!\!\!\!\!\!\!\! +\, \dfrac{1}{NT_k} \displaystyle\sum_{t\,=\,\tau+1}^{T_k}\displaystyle\sum_{j\,=\,1}^{N} \dotp{ G_{j,x}(z_{j,k}(t);\xi_{k,t}) }{ x_{j,k}(t) - {\color{black}\underbar{$x$}_k(t)} }
	\\[0.2cm]
	&  &\!\!\!\!\!\!\!\!\!\!\!\!\!\!\!\!\!\!\!\!\!\!\!\! +\, \dfrac{1}{T_k}\displaystyle\sum_{t\,=\,\tau+1}^{T_k}\dotp{\dfrac{1}{N}\displaystyle\sum_{j\,=\,1}^{N} G_{j,x}(z_{j,k}(t);\xi_{k,t}) }{ {\color{black}\underbar{$x$}_k(t)} - x^\star}.
	\ea
	\]
	On the right-hand side of the above inequality, the third term can be bounded by applying the Cauchy-Schwartz inequality and the gradient boundedness; for the second term and the fourth term, application of Lemma~\ref{lem:online-bound} yields,
	\[
	\ba{c}
	 \!\!\!\!\!\!\!\!\!\!\!\!\!\!\!\!\!\!\!\!\!\!\!\!\!\!\!\!\!\!\!\!\!\!\!\!\!\!\!\!\!\!\!\!\!\! \displaystyle\sum_{t\,=\,\tau+1}^{T_k}\!\dotp{ G_{j,y}(z_{j,k}(t);\xi_{k,t}) }{ \hat{y}_{j,k}^\star - y_{j,k}(t)} 
	 \\[0.2cm]
	 \,\leq\,
	  \dfrac{\l\Vert \hat{y}_{j,k}^\star - y_{j,k}(\tau\!+\!1)  \r\Vert^2 }{2\eta_k} 
	  \,+\,
	  \dfrac{\eta_k}{2}\displaystyle\sum_{t\,=\,\tau+1}^{T_k} \l\Vert G_{j,y}(z_{j,j}(t);\xi_{k,t})\r\Vert^2
	  \ea
	\]
	\[
	\ba{c}
	 \!\!\!\!\!\!\!\!\!\!\!\!\!\!\!\!\!\!\!\!\!\!\!\!\!\!\!\!\!\!\!\!\!\!\!\!\! \displaystyle\sum_{t\,=\,\tau+1}^{T_k}\!\dotp{\dfrac{1}{N}\displaystyle\sum_{j\,=\,1}^{N} G_{j,x}(z_{j,k}(t);\xi_{k,t}) }{ {\color{black}\underbar{$x$}_k(t)} - x^\star}
	 \\[0.2cm]
	 \leq
	  \dfrac{\l\Vert {\color{black}\underbar{$x$}_k(\tau\!+\!1)}\!-\!x^\star \r\Vert^2}{2\eta_k} \!+\!
	  \dfrac{\eta_k}{2} \!\!\displaystyle\sum_{t\,=\,\tau+1}^{T_k} \bigg\Vert \dfrac{1}{N}\displaystyle\sum_{j\,=\,1}^{N} G_{j,x}(z_{j,k}(t);\xi_{k,t}) \bigg\Vert^2
	 \ea
	\]
	which allows us to bound $S_3+S_5$ by 
	
	\[
	\ba{rcl}
	S_3\,+\,S_5
	&\!\!\!\! \leq \!\!\!\!& \dfrac{2{\color{black}r}L+\sqrt{2}{\color{black}c}}{NT_k}\!\!\!\displaystyle\sum_{t\,=\,\tau+1}^{T_k}\displaystyle\sum_{j\,=\,1}^{N}\left\Vert z_{j,k}(t-\tau)-z_{j,k}(t) \right\Vert
	\\[0.2cm]
	&  &\!\!\!\!\!\!\!\!\!\!\!\!\!\!\!\!\!\!\!\!\!\!\!\!\!\!\!\!\!\!\!\!\! +\, \dfrac{1}{2\eta_k T_k}\! \bigg(\!\!\l\Vert {\color{black}\underbar{$x$}_k(\tau+1)} \!-\!x^\star \r\Vert^2\!+\!\dfrac{1}{N}\!\displaystyle\sum_{j\,=\,1}^{N} \! \l\Vert \hat{y}_{j,k}^\star\!-\!y_{j,k}(\tau+1)  \r\Vert^2\!\!\bigg)
	\\[0.2cm]
	&  &\!\!\!\!\!\!\!\!\!\!\!\!\!\!\!\!\!\!\!\!\!\!\!\!\!\!\!\!\!\!\!\!\! +\, \dfrac{{\color{black}c}}{NT_k}\displaystyle\sum_{t\,=\,\tau+1}^{T_k}\displaystyle\sum_{j\,=\,1}^{N} \l\Vert x_{j,k}(t) -{\color{black}\underbar{$x$}_k(t)} \r\Vert \,+\, \eta_k {\color{black}c}^2.
	\ea
	\]
	
	Taking expectation of $S_3+S_5$ and adding previous bounds on $\mathbb{E}[ S_1 ]$, $\mathbb{E}[ S_2 ]$, and $\mathbb{E}[ S_4 ]$ to it lead to the final bound on $\mathbb{E}[\,\text{PDG}^+\,]$. The proof is completed by adding the established upper bounds on $\mathbb{E}[ \text{PDG}^+ ]$ and $\mathbb{E}[ \text{PDG}^- ]$.
\end{proof}

Lemma~\ref{lem:boundII} is based on the ergodic analysis of the mixing process. We may set $\tau=0$ to take the traditional stochastic gradient method with i.i.d.\ samples. \tc{black}{In fact, by combining results of Lemmas~\ref{lem:decomposegap} and~\ref{lem:boundII}, we can obtain a loose bound $O(1/\sqrt{T_k})$ bound for the surrogate gap $\text{Err}'(\hat{x}_{i,k},\hat{y}_k)$.} Instead, we combine the strong concavity of $\psi_j(x,y_j)$ in terms of $y_j$ with Lemma~\ref{lem:boundII} to establish $O(1/{T_k})$ bound on the surrogate gap.
\begin{lemma}\label{lem:recursivesurrogate}
	Let Assumptions~\ref{as:domain}--\ref{as:grad} hold. For $\eta_k$ and $T_k$ that satisfy $L_x \eta_k T_k\geq 16$ and $T_k\geq1+\lceil\,\log(\Gamma T_k)/|\log\rho|\,\rceil=\tau$, we have
	\be
	\non
	\ba{rcl}
	\mathbb{E}[\,\text{Err}'(\hat{x}_{i,k},\hat{y}_k)\,]
	&\!\!\!\!\leq\!\!\!\!& 
	4\,\mathbb{E}[\,\normalfont\text{NET}'\,]
	\,+\,
	\dfrac{{4}\,{\color{black}c}^2}{{L_y}} \l(\dfrac{4}{\sqrt{T_k-\tau}} +\dfrac{1}{ T_k}\r)^2
	\\[0.3cm]
	& & \!\!\!\!\!\!\!\!\!\!\!\!\!\!\!\!\!\!\!\!\!\!\!\!\!\!\!\!\!\!\!\!\!\!\!\!\!\!\!\!\!\!\!\!\!\!\!\!\!\!
	\,+\,
	16\eta_k {\color{black}c}({\color{black}r}L+{\color{black}c})\tau
	+
	4\eta_k {\color{black}c}^2
	+
	 \dfrac{8\eta_k {\color{black}c}^2 \tau^2}{T_k}
	 +
	  \dfrac{8{\color{black}rc}(\tau +1) }{T_k} 
	\\[0.3cm]
	&  & 
	\!\!\!\!\!\!\!\!\!\!\!\!\!\!\!\!\!\!\!\!\!\!\!\!\!\!\!\!\!\!\!\!\!\!\!\!\!\!\!\!\!\!\!\!\!\!\!\!\!\!
	\,+\,
	\dfrac{16\,\mathbb{E}[\,\text{Err}'(\hat{x}_{i,k-1},\hat{y}_{k-1})\,]}{L_y\eta_kT_k}
	\,+\,
	\displaystyle\sum_{j\,=\,1}^{N}  \dfrac{8\,\mathbb{E}[\,\text{Err}'(\hat{x}_{j,k-1},\hat{y}_{k-1})\,]}{L_x \eta_k NT_k}
	\ea
	\ee
	where
	 \[
	\ba{rcl}
	\normalfont\text{NET}' 
	&\!\!\! \DefinedAs \!\!\! & 
	\normalfont\text{NET}
	\,+\,
	\dfrac{4\,({\color{black}r}L+{\color{black}c})}{NT_k} \displaystyle\sum_{t\,=\,1}^{T_k}\displaystyle\sum_{j\,=\,1}^{N}  \Vert x_{j,k}(t)- {\color{black}\underbar{$x$}_k(t)} \Vert .
	\ea
	\]
\end{lemma}
\begin{proof}
	Taking expectation of~\eqref{eq.cc3} yields,
	\be
	\non
	\ba{rcl}
	\mathbb{E}[\,\text{Err}'(\hat{x}_{i,k},\hat{y}_k)\,]
	&\!\! \geq\!\!  &
	  \dfrac{L_y}{2N} \displaystyle\sum_{j\,=\,1}^{N} \mathbb{E}[\,\Vert y_j^\star-\hat{y}_{j,k}^\star \Vert^2\,]
	  \\[0.2cm]
	  &\!\! \geq\!\! &
	  \dfrac{L_y}{2}\, \bigg(\dfrac{1}{N} \displaystyle\sum_{j\,=\,1}^{N} \mathbb{E}[\,\Vert y_j^\star-\hat{y}_{j,k}^\star \Vert\,]\bigg)^2
	  \ea
	\ee
	where we use $\mathbb{E}[ X^2 ]\geq\mathbb{E}[ X ]^2$ along with $(1/N )\sum_{i\,=\,1}^{N} a_i^2\geq( (1/N) \sum_{i\,=\,1}^N{a_i})^2$ to establish the second inequality. Substituting the above inequality into the result in Lemma~\ref{lem:boundII} and combining it with Lemma~\ref{lem:decomposegap} yield,
	\be
	\non
	\ba{rcl}
	\expect{ \,\text{Err}'(\hat{x}_{i,k},\hat{y}_k)\, } 
	&\!\!\!\leq\!\!\!& 
	\mathbb{E}[\,\text{NET}\,]
	\,+\, \mathbb{E}[\,\normalfont\text{MIX}\,]
	\; + \;
	 \dfrac{\sqrt{2}\,{\color{black}c}}{\sqrt{L_y}}\l(\dfrac{4}{\sqrt{T_k\,-\,\tau}}+\dfrac{1}{ T_k}\r) \mathbb{E}[\, \text{Err}'(\hat{x}_{i,k},\hat{y}_k)\, ]^{1/2}.
	\ea
	\ee
\tc{black}{Thus, $ \zeta \DefinedAs \mathbb{E}[\,\text{Err}'(\hat{x}_{i,k},\hat{y}_k)\,]^{1/2}$ satisfies the quadratic inequality $\zeta^2 \leq a \zeta+b$. Combining the values of $\zeta$ that satisfy this inequality with $(\Vert a \Vert + \Vert b \Vert)^2 \leq 2 \Vert a \Vert^2+2\Vert b \Vert^2$ leads to,}
	\be
	\label{eq.keybound}
	\ba{rcl}
	\mathbb{E}[\,\text{Err}'(\hat{x}_{i,k},\hat{y}_k)\,]
	&\!\! \leq\!\! & 
	2\,\mathbb{E}[\,\text{NET}\,]
	\,+\,
	 2\,\mathbb{E}[\,\normalfont\text{MIX}\,]
	\,+\,
	\dfrac{{2}{\color{black}c}^2}{{L_y}}\! \l(\dfrac{4}{\sqrt{T_k-\tau}} +
	\dfrac{1}{ T_k}\r)^2.
	\ea
	\ee
	
	\tc{black}{The remaining task is to express expectations on the right-hand side of~\eqref{eq.keybound} in terms of previously introduced terms. Lemma~\ref{lem.netavearaging} provides a bound on $\mathbb{E}[\,\text{NET}\,]$ and, next, we evaluate the terms in $\mathbb{E}[\,\text{MIX}\,]$. Combination of the triangle inequality with non-expansiveness of projection allows us to bound $\mathbb{E}[\,\left\Vert z_{j,k}(t-\tau)-z_{j,k}(t) \right\Vert\,] $ by} 
	\[
	\mathbb{E}[\,\Vert y_{j,k}'(t-\tau)-y_{j,k}'(t) \Vert\, ]
	\,+\,
	 \mathbb{E}[\, \Vert x_{j,k}(t-\tau)-x_{j,k}(t) \Vert\, ]
	\]
	where $\mathbb{E}[ \,\left\Vert x_{j,k}(t-\tau)-x_{j,k}(t) \right\Vert \,]$ can be further bounded by a sum of three terms: $\mathbb{E}[\, \left\Vert x_{j,k}(t-\tau)-{\color{black}\underbar{$x$}_k(t-\tau)} \right\Vert \,] $, $\mathbb{E}[\,\left\Vert {\color{black}\underbar{$x$}_k(t-\tau)} - {\color{black}\underbar{$x$}_k(t)} \right\Vert \,] $, and $\mathbb{E}[\, \left\Vert {\color{black}\underbar{$x$}_k(t)}-x_{j,k}(t) \right\Vert \,] $. Similarly, $\left\Vert {\color{black}\underbar{$x$}_k(t-\tau) }- {\color{black}\underbar{$x$}_k(t)} \right\Vert \leq \left\Vert \bar{x}_k'(t-\tau)-\bar{x}_k'(t) \right\Vert$.  By the gradient boundedness, it is clear from the primal-dual updates that $\mathbb{E}[\,\Vert y_{j,k}'(t-\tau)-y_{j,k}'(t) \Vert \,],\mathbb{E}[\,\Vert {\color{black}\underbar{$x$}_k(t-\tau)} - {\color{black}\underbar{$x$}_k(t)} \Vert \,]  \leq \eta_k {\color{black}c} \tau$. Thus, we have
	\[
	\ba{rcl}
	\!\!& &\!\!\!\!\!\!\!\!\!\!\!\!  \dfrac{1}{NT_k}\!\displaystyle\sum_{t\,=\,\tau+1}^{T_k}\displaystyle\sum_{j\,=\,1}^{N} \mathbb{E}[\,\Vert z_{j,k}(t-\tau)-z_{j,k}(t) \Vert \,]
	\\[0.2cm]
	\!\!&\!\!\!\!\leq\!\!\!\!& 
	2\eta_k {\color{black}c}\tau + \dfrac{1}{NT_k}\!\displaystyle\sum_{t\,=\,\tau+1}^{T_k}\displaystyle\sum_{j\,=\,1}^{N} \mathbb{E}[\,\Vert x_{j,k}(t\!-\!\tau)\!-\!{\color{black}\underbar{$x$}_k(t\!-\!\tau)} \Vert\,] 
	\\[0.2cm]
	&& +\, \dfrac{1}{NT_k}\displaystyle\sum_{t\,=\,\tau+1}^{T_k}\displaystyle\sum_{j\,=\,1}^{N} \mathbb{E}[\,\Vert {\color{black}\underbar{$x$}_k(t)}-x_{j,k}(t)\Vert\,]
	\\[0.2cm]
	\!\!&\!\!\!\!\leq\!\!\!\!& 
	2\eta_k {\color{black}c} \tau 
	\,+\, \dfrac{2}{NT_k}\displaystyle\sum_{t\,=\,1}^{T_k}\displaystyle\sum_{j\,=\,1}^{N} \mathbb{E}[\,\Vert x_{j,k}(t)- {\color{black}\underbar{$x$}_k(t)} \Vert\,]
	\ea
	\] 
where we \tc{black}{sum over $t$ from $1$ to $T_k$ instead of from $\tau+1$ to $T_k$} in the last inequality.
	Now, we turn to next two expectations as follows,
	\[
	\ba{rcl}
	&& \!\!\!\!\!\!\!\!\! \mathbb{E}[\,\Vert x^\star -  {\color{black}\underbar{$x$}_k(\tau+1)} \Vert^2\,]
	\\[0.2cm]
	& \!\! =\!\!  & \mathbb{E}[\,\Vert  {\color{black}\underbar{$x$}_k(\tau+1)} - {\color{black}\underbar{$x$}_k(1)} +  {\color{black}\underbar{$x$}_k(1)} -  x^\star \Vert^2\,]
	\\[0.2cm]
	& \!\! \leq \!\! & 2\,\mathbb{E}[\,\Vert  {\color{black}\underbar{$x$}_k(\tau+1)} -  {\color{black}\underbar{$x$}_k(1)} \Vert^2\,] \,+\, 2\,\mathbb{E}[\,\Vert  {\color{black}\underbar{$x$}_k(1)} - x^\star \Vert^2\,]
	\\[0.2cm]
	& \!\! \leq \!\!  & 2\,\mathbb{E}[\,\Vert \bar{x}_k'(\tau+1) -\bar{x}_k'(1) \Vert^2\,] \,+\, 2\,\mathbb{E}[\,\Vert  {\color{black}\underbar{$x$}_k(1)}-x^\star \Vert^2\,]
	\\[0.1cm]
	& \!\! \leq \!\! & 2 \,\eta_k^2\, {\color{black}c}^2\, \tau^2  \,+\,  
	\dfrac{2}{N} 
	\displaystyle \sum_{j\,=\,1}^N\mathbb{E}[\,\Vert x_{j,k}'(1)-x^\star\Vert^2\,]
	\\[0.2cm]
	& \!\! = \!\! & 2\, \eta_k^2\, {\color{black}c}^2\, \tau^2  \,+\,  \dfrac{2}{N} \displaystyle \sum_{j\,=\,1}^N\mathbb{E}[\,\Vert x_{j,k}(1)-x^\star\Vert^2\,]
	\ea
	\]
	where we apply the inequality $\norm{a+b}^2 \leq 2\norm{a}^2+2\norm{b}^2$ and the non-expansiveness of projection for the first and the second inequalities; {\color{black}the third inequality is because of~\eqref{eq:simple-update}, the gradient boundedness, and application of  the Jensen's inequality to $\Vert\cdot\Vert^2$ with $\underbar{$x$}_{k}(1) \DefinedAs \frac{1}{N}\sum_{j=1}^{N}x_{j,k}'(1)$; and the last equality follows the initialization $x_{j,k}'(1) = x_{j,k}(1)$.}
	Similarly, we derive  
	\[
	\ba{rcl}
	\!\!\!\!&&\!\!\!\!\!\!\!\!\! \mathbb{E}[\,\Vert \hat{y}_{j,k}^\star -y_{j,k}(\tau+1)  \Vert^2\,]
	\\[0.2cm]
	\!\!\!\!&\!\! =\!\! & \mathbb{E}[\,\Vert \hat{y}_{j,k}^\star  - y_{j,k}(1) + y_{j,k}(1) - y_{j,k}(\tau+1)  \Vert^2\,]
	\\[0.2cm]
	\!\!\!\!& \!\! \leq\!\!  & 2\,\mathbb{E}[\,\Vert \hat{y}_{j,k}^\star -y_{j,k}(1) \Vert^2\,]\, +\,2\, \mathbb{E}[\,\Vert y_{j,k}(1)- y_{j,k}(\tau)\Vert^2\,]
	\\[0.2cm]
	\!\!\!\!& \!\! \leq\!\! & 2\,\mathbb{E}[\,\Vert \hat{y}_{j,k}^\star -y_{j,k}(1) \Vert^2\,] \,+\,2\, \mathbb{E}[\,\Vert y_{j,k}'(1)- y_{j,k}'(\tau)  \Vert^2\,]
	\\[0.2cm]
	\!\!\!\!& \!\!\!  \leq \!\!\!  & 2\,\mathbb{E}[\,\Vert \hat{y}_{j,k}^\star -y_{j,k}(1) \Vert^2\,] \,+\,2 \,\eta_k^2 {\color{black}c}^2 \tau^2
	\\[0.2cm]
	\!\!\!\!& \!\!  = \!\! & 2\,\mathbb{E}[\,\Vert \hat{y}_{j,k}^\star-y_{j}^\star + y_{j}^\star - y_{j,k}(1) \Vert^2\,] \,+\,2\, \eta_k^2 {\color{black}c}^2 \tau^2
	\\[0.2cm]
	\!\!\!\!& \!\! \leq \!\! & 4\,\mathbb{E}[\,\Vert \hat{y}_{j,k}^\star-y_{j}^\star \Vert^2\,]\,+\,4\,\mathbb{E}[\,\Vert y_{j}^\star- y_{j,k}(1) \Vert^2\,] \,+\, 2\, \eta_k^2 {\color{black}c}^2 \tau^2.
	\ea
	\]
	From~\eqref{eq.cc3} we know that
	\[
	\dfrac{1}{N}\displaystyle\sum_{j\,=\,1}^{N}\mathbb{E}[\,\Vert \hat{y}_{j,k}^\star -y_{j}^\star \Vert^2\,]
	\,\leq\, \frac{2}{L_y}\,\mathbb{E}[\,\text{Err}'(\hat{x}_{i,k},\hat{y}_k)\,].
	\] 
	Now, we collect above inequalities for~\eqref{eq.keybound}. \tc{black}{For notational convenience, we sum over $t$ from $1$ until $T_k$ instead of from $\tau+1$ to $T_k$ and combine similar terms to obtain the following bound on $\mathbb{E}[\,\normalfont\text{MIX}\,]$,}
	\be
	\non
	\ba{rcl}
	& & \!\!\!\!\!\!\!\!\!\!\!\!\!\!
	4\eta_k {\color{black}c}({\color{black}r}L + {\color{black}c})\tau
	+\dfrac{4({\color{black}r}L + {\color{black}c})}{NT_k} \displaystyle\sum_{t\,=\,1}^{T_k}\displaystyle\sum_{j\,=\,1}^{N}\!   \expect{\, \left\Vert x_{j,k}(t) \!-\! {\color{black}\underbar{$x$}_k(t)} \right\Vert \,} 
	\\[0.2cm]
	&  & \!\!\!\!\!\!\!\!\!\!\!\!\!\!
	+\,
	\dfrac{1}{\eta_k NT_k} \displaystyle\sum_{j\,=\,1}^{N}  \big(2\mathbb{E}[\,\Vert y_{j}^\star-y_{j,k}(1)\Vert^2\,] 
	+ \mathbb{E}[\,\Vert x_{j,k}(1)-x^\star \Vert^2\,]\big)
	\\[0.3cm]
	&& \!\!\!\!\!\!\!\!\!\!\!\!\!\!
	+\,
	\dfrac{4\,\mathbb{E}[\,\text{Err}'(\hat{x}_{i,k},\hat{y}_k)\,]}{L_y \eta_kT_k } \,+\, \dfrac{2 \eta_k\, {\color{black}c}^2 \tau^2}{T_k} \,+\, \dfrac{2{\color{black}r c}\,(\tau+1) }{T_k} \,+\,\eta_k\, {\color{black}c}^2.
	\ea
	\ee
	
	We note the restarting scheme of Algorithm~\ref{alg.main}, i.e., $x_{j,k}(1) = \hat{x}_{j,k-1}$ and $y_{j,k}(1) = \hat{y}_{j,k-1}$. By~\eqref{eq.cc1},~\eqref{eq.cc2}, and Lemma~\ref{lem:objective2PD}, we have   
	\[
	\ba{rcl}
	\mathbb{E}[\,\Vert x_{j,k}(1)-x^\star \Vert^2\,]
	&\!\! =\!\! & 
	\mathbb{E}[\,\Vert \hat{x}_{j,k-1}-x^\star \Vert^2\,]
	\\[0.2cm]
	&\!\! \leq\!\! & 
	\dfrac{2}{L_x}\, \mathbb{E}[\,\text{Err}(\hat{x}_{j,k-1})\,]
	\\[0.25cm]
	&\!\! \leq\!\! & 
	\dfrac{2}{L_x} \,\mathbb{E}[\,\text{Err}'(\hat{x}_{j,k-1},\hat{y}_{k-1})\,]
	\ea
	\]
	\[
	\ba{rcl}
	\dfrac{1}{N}\displaystyle\sum_{j\,=\,1}^{N}\mathbb{E}[\,\Vert y_{j}^\star - y_{j,k}(1) \Vert^2\,]
	&\!\! =\!\! & 
	\dfrac{1}{N}\displaystyle\sum_{j\,=\,1}^{N}\mathbb{E}[\,\Vert y_{j}^\star - \hat{y}_{j,k-1}\Vert^2\,]
	\\[0.2cm]
	&\!\! \leq\!\! &
	\dfrac{2}{L_y} \,\mathbb{E}[\,\text{Err}'(\hat{x}_{i,k-1},\hat{y}_{k-1})\,]. 
	\ea
	\] 
	Combining above inequalities with the bound on $\mathbb{E}[\,\normalfont\text{MIX} \,]$ yields a bound on $\mathbb{E}[\, \text{Err}'(\hat{x}_{i,k},\hat{y}_k) \,]$. Finally, we utilize \tc{black}{$L_x \eta_k T_k \geq 16$} to finish the proof. 
\end{proof}

	\vspace*{-2ex}
\subsection{Proof of Theorem~\ref{thm.main}}
\label{sec.thm}

The proof is based on the result in Lemma~\ref{lem:recursivesurrogate}. We leave the mixing time $\tau$ to be determined so that it works for every round $k$ and focus on the averaged surrogate gap $\overline{\text{Err}}_k' \DefinedAs\frac{1}{N}\sum_{j\,=\,1}^{N}\mathbb{E}[\,\text{Err}'(\hat{x}_{j,k},\hat{y}_k)\,]$. 

Let
	\[
	H_k 
	\,\DefinedAs\, 
	\dfrac{4\,{\color{black}c}^2}{L_y} \l(\dfrac{{\color{black}8}}{\sqrt{T_k-\tau}}  + \dfrac{1}{ T_k}\r)^2
	\]  
	\[
	\ba{rcl}
	E_k 
	& \!\! \DefinedAs\!\!  & 
	{\color{black}8\, ( 4{\color{black}r}L  +  6{\color{black}c} ) \,\Delta_k}
	\,+\,
	16\, \eta_k {\color{black}c}\, ({\color{black}r}L + {\color{black}c})\,\tau
	\,+\,\dfrac{8{\color{black}r c}\,(\tau + 1) }{T_k} \,+\, 4\,\eta_k {\color{black}c}^2.
	\ea
	\] 
	We first apply Lemma~\ref{lem.netavearaging} to obtain {\color{black}$\mathbb{E}[\,\normalfont\text{NET}'\,] \leq 2(4{\color{black}r}L + 6{\color{black}c}) \Delta_k$} for each round $k$. With previous simplified notation, we apply this inequality for the bound in Lemma~\ref{lem:recursivesurrogate}, and then take average over $i=1,\ldots,N$ on both sides to obtain,
	\be
	\label{eq:one-inner-interation-bound2}
	\overline{\text{Err}}_k'  
	\,\leq\,
	  H_k 
	  \,+\, 
	  E_k
	\,+\,
	\dfrac{ 8\,\eta_k\, {\color{black}c}^2\, \tau^2}{T_k} 
	\,+\,
	 \dfrac{\overline{\text{Err}}_{k-1}'}{\tc{black}{L'} \,\eta_k\, T_k}  
	\ee
	where $1/{\color{black}L'}\DefinedAs 16/L_y+8/L_x$.
	In Algorithm~\ref{alg.main}, we have updates $\eta_k = \eta_{k-1}/2$ and $T_k = 2T_{k-1}$. Since $\eta_1\geq 4/{\color{black}L'}$ and $T_1\geq 1$, we have 
	${\color{black}L'} \eta_kT_k\geq 4$ for all $k\geq 1$. Clearly, the assumption $L_x \eta_k T_k\geq 16$ holds in Lemma~\ref{lem:recursivesurrogate}. Therefore,~\eqref{eq:one-inner-interation-bound2} can be simplified into
	\be
	\label{eq:one-inner-interation-bound3}
	\overline{\text{Err}}_{k}' 
	\,\leq\,
	 H_k 
	 \,+\,
	  E_k  
	  \,+\,
	   \dfrac{ 8\, \eta_k\, {\color{black}c}^2\, \tau^2}{T_k} 
	   \,+\,
	    \frac{\overline{\text{Err}}_{k-1}'}4.
	\ee
	Simple comparisons show that $H_k\leq H_{k-1}/2$ and $E_k \geq E_{k-1}/2$. Thus, both $\eta_k/T_k= (1/4)^{k-1}\eta_1/T_1$ and $H_k\leq H_{1}/2^{k-1}$ are true for all $k\geq 1$. If we set $\tau = 1+\lceil\,\log(\Gamma T)/|\log\rho|\,\rceil\leq T_1$ with suitable $T_1$ and $K$, Lemma~\ref{lem:recursivesurrogate} applies at any round $k$. 
	Starting from the final round $k=K$, we repeat~\eqref{eq:one-inner-interation-bound3} to obtain,
	\[
	\ba{rcl}
	\!\!\overline{\text{Err}}_{K}' &\!\! \leq\!\! & H_K\,+\,E_K \,+\, \dfrac{2}{4^{K-2}}\dfrac{\eta_1\, {\color{black}c}^2\, \tau^2}{T_1} \,+\, \dfrac{\overline{\text{Err}}_{K-1}'}4
	\\[0.3cm]
	&\!\! \leq \!\! & \dfrac{1}{2^{K-2}}\displaystyle\sum_{k\,=\,1}^K\l(\frac12\r)^k\! H_{1} \,+\,\displaystyle\sum_{k\,=\,1}^K\l(\frac12\r)^k \!E_{K} 
	\\[0.3cm]
	&& \,+\, \dfrac{2K}{4^{K-2}}\dfrac{\eta_1\, {\color{black}c}^2\, \tau^2}{T_1} \,+\, \l(\dfrac14\r)^{K}\!\!\overline{\text{Err}}_0'
	\\[0.3cm]
	&\!\! \leq \!\! & \dfrac{8\,H_1 T_1}{T} \,+\, 2E_K \,+\,\dfrac{32\, K\, \eta_1 \,{\color{black}c}^2 \,\tau^2 \,T_1}{T^2} \,+\, \dfrac{{\color{black}r}^2 T_1^2}{T^2}
	\ea
	\]
	where $T=\sum_{k\,=\,1}^{K}T_k = (2^K-1)T_1\leq 2^K T_1$ and $\overline{\text{Err}}_0'\leq\overline{\text{Err}}_{0}=\sup_{x\,\in\,\mathcal X,y \,\in\, \mc Y}\|(x,y)\|^2\leq {\color{black}r}^2$.
	We now substitute $H_1$ and $E_K$ into this bound and bound $\overline{\text{Err}}_{K}' $ by 
	\[
	\begin{array}{l}
	 \dfrac{32\, {\color{black}c}^2 T_1}{L_y \,T} \l(\dfrac{ {\color{black}8}}{\sqrt{T_1-\tau}}  + \dfrac{1}{ T_1}\r)^2 \,+\, {\color{black}16\,(4{\color{black}r}L + 6{\color{black}c}) \, \Delta_K}
	\\[0.3cm]
	 \,+\, \dfrac{32 \,K\eta_1\, {\color{black}c}^2\, \tau^2\, T_1}{T^2}  \,+\, \dfrac{{\color{black}r}^2 T_1^2}{T^2} \,+\, \dfrac{16\,{\color{black}rc}\,(\tau + 1)}{T_K} 
	\\[0.3cm]
	\,+\, 32\,\eta_K {\color{black}c}\, (2{\color{black}r}L + \sqrt{2}{\color{black}c})\tau \,+\,  8\,\eta_K {\color{black}c}^2.
	\end{array}
	\]
	
	Since $T\leq 2^K T_1$, we have $\eta_K=\eta_1/2^K\leq \eta_1T_1/T$ and $T_K=2^{K-1}T_1\geq T/2$. Since $T = (2^K-1)T_1\geq 2^{K-1} T_1$, we have $K\leq 1+\log(T/T_1)$. Thus, $\sum_{k\,=\,1}^{K}\eta_kT_{k}=K\eta_1 T_1\leq  \eta_1T_1 (1+\log(T/T_1))$. Therefore, the above bound has the following order,
	\[
	C_1\,\dfrac{{\color{black}c}\,({\color{black}c}+{\color{black}r}L)\log^2(\sqrt{N}\,T)}{T\,(1 - \sigma_2(W))} \,+\, C_2\,\dfrac{ {\color{black}c}\,({\color{black}c} + {\color{black}r}L) (1 +  T_1)}{T}
	\]
	where $C_1$ are $C_2$ are absolute constants.  
	Since $\overline{\text{Err}}_{K}\leq\overline{\text{Err}}_{K}'$, it also bounds $\overline{\text{Err}}_{K}$. {\color{black}The proof is completed by combining $\eta_1\geq 4/L'$ with $T_1\geq \tau$.}

	\vspace*{-2ex}
\section{Computational experiments}\label{sec.comp}

\tc{black}{We first utilize a modified Mountain Car Task~\cite[Example~10.1]{sutton2018reinforcement} for multi-agent policy evaluation problem.} We generate the dataset using the approach presented in~\cite{wai2018multi}, obtain a policy by running SARSA with $d=300$ features, and sample the trajectories of states and actions according to the policy. The discount factor is set to $\gamma=0.95$. We simulate the communication network with $N$ agents using the Erd\H{o}s-R\'{e}nyi graph with connectivity $0.1$. At every time instant, each agent observes a local reward that is generated as a random proportion of the total reward. Since the stationary distribution $\Pi$ is unknown, we use sampled averages from the entire dataset to compute sampled versions $\hat{A}$, $\hat{b}$, and $\hat{C}$ of $A$, $b$, and $C$. We then formulate an empirical MSPBE as $(1/2) \Vert \hat{A}x-  \hat{b}\Vert^2_{\hat{C}^{-1}}$ and compute the optimal empirical MSPBE. We use this empirical MSPBE as an approximation of the population MSPBE to calculate the optimality gap. The dataset contains $85453$ samples and we run our online algorithm over one trajectory of $30000$ samples using multiple passes. \tc{black}{We set an initial restart time to $T_1=10^5$ and a restart round to $K=4$ to insure $T_1 \simeq O(K+\log T_1)$, take large bounds for Euclidean projections, and choose different learning rates $\eta$.}  

We compare \tc{black}{the performance of} Algorithm~\ref{alg.main} (DHPD) with stochastic primal-dual (SPD) algorithm under different settings. For $N=1$, SPD corresponds to GTD algorithm in~\cite{liu2015finite, wang2017finite, touati2017convergent}, and for $N>1$, SPD becomes the multi-agent GTD algorithm~\cite{lee2018primal}. We show computational results for $N=1$ and $N=10$ in Fig.~\ref{fig.compare1} and Fig.~\ref{fig.compare100}, respectively. The optimality gap is the difference between empirical MSPBE and the optimal one. \tc{black}{Our algorithm achieves a smaller optimality gap} than SPD in all cases, thereby demonstrating its sample efficiency. \tc{black}{In computational experiments with simple diminishing stepsizes, the algorithm converges relatively slow as is typical in stochastic optimization. Also, our online algorithm competes well against the approach that utilizes pre-collected i.i.d.\ samples (instead of true i.i.d.\ samples from the stationary distribution) in a fixed buffer. }

\tc{black}{In our second computational experiment, we test randomly generated multi-agent MDPs for a ring network with $N = 5$ agents that utilize a fixed policy~\cite[Example~1]{lee2018primal}. This leads to a Markov chain with $\mathcal{S} = \{1, 2, 3, 4 \}$, $\gamma=0.95$, $\phi(s) = [\, \phi_1(s) \;\, \phi_2(s) \;\, \phi_3(s)\;\, \phi_4(s) \,]^T \in \mathbb{R}^4$ where $\phi_i(s) = {\rm e}^{-{(s-i)}^2}$, and ${R}_j(s) = \mathds{1}({ s = 4})$. In Fig.~\ref{fig.mmdps}, we demonstrate that our online algorithm with an initial restart time $T_1=20000$ and a restart round $K=3$ outperforms SPD algorithm that utilizes diminishing stepsizes or a replay buffer.}

\begin{figure}[]
	\centering
	\begin{tabular}{cc}
		\rotatebox{90}{\quad\quad\quad\quad\quad\quad\quad\small Optimality gap} 
		&
		\hspace*{-0.3cm} 
		{\label{} \includegraphics[scale=0.6]{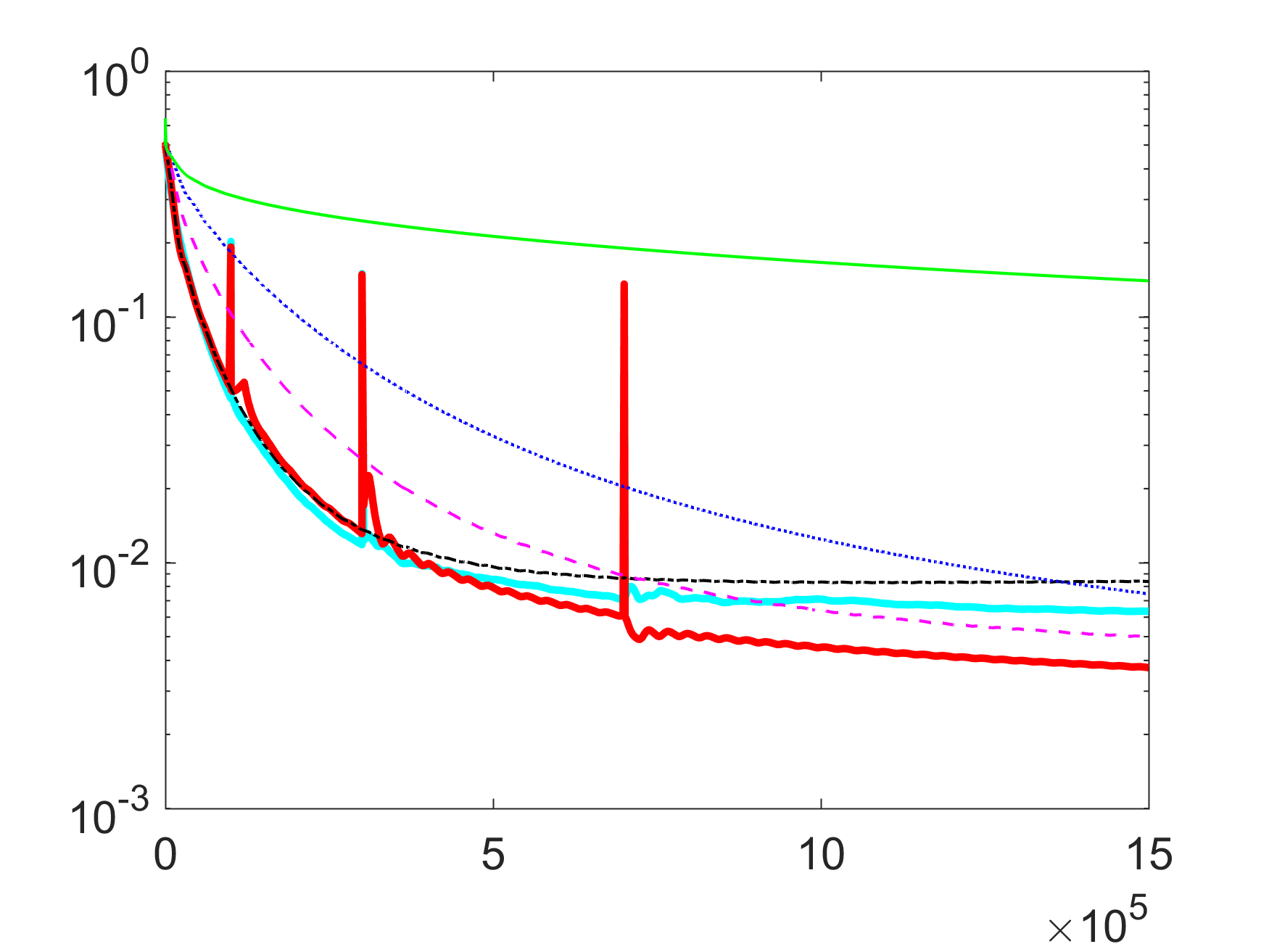}}
		\\
		{} 
		& { Number of iterations}
	\end{tabular}
	\caption{ {\color{black}Performance} comparison for the centralized problem with $N=1$. Our algorithm with stepsize $\eta = 0.1$ and $K=4$ ({\large\color{red}$\textbf{--\!--}$}) \tc{black}{achieves a smaller} optimality gap than stochastic primal-dual algorithm with stepsize: $\eta=0.1$ ($\text{-}$$\cdot$$\text{-}$), $\eta=0.05$ ({\color{magenta}$\text{-}\text{-}\text{-}$}), $\eta=0.025$ ({\color{black}$\cdot$$\cdot$$\cdot$$\cdot$}), and $\eta = 1/\sqrt{t}$ ({\color{green}$\text{--\!--}$}). {\color{black}It also provides a smaller optimality gap than the approach that utilizes pre-collected} i.i.d.\ samples in a buffer  ({\large\color{figcolor1}$\textbf{--\!--}$}). }
	\label{fig.compare1}
\end{figure}

\begin{figure}[]
	\centering
	\begin{tabular}{cc}
		\rotatebox{90}{\quad\quad\quad\quad\quad\quad\quad\small Optimality gap} 
		&
		\hspace*{-0.3cm} 
				{\label{} \includegraphics[scale=0.6]{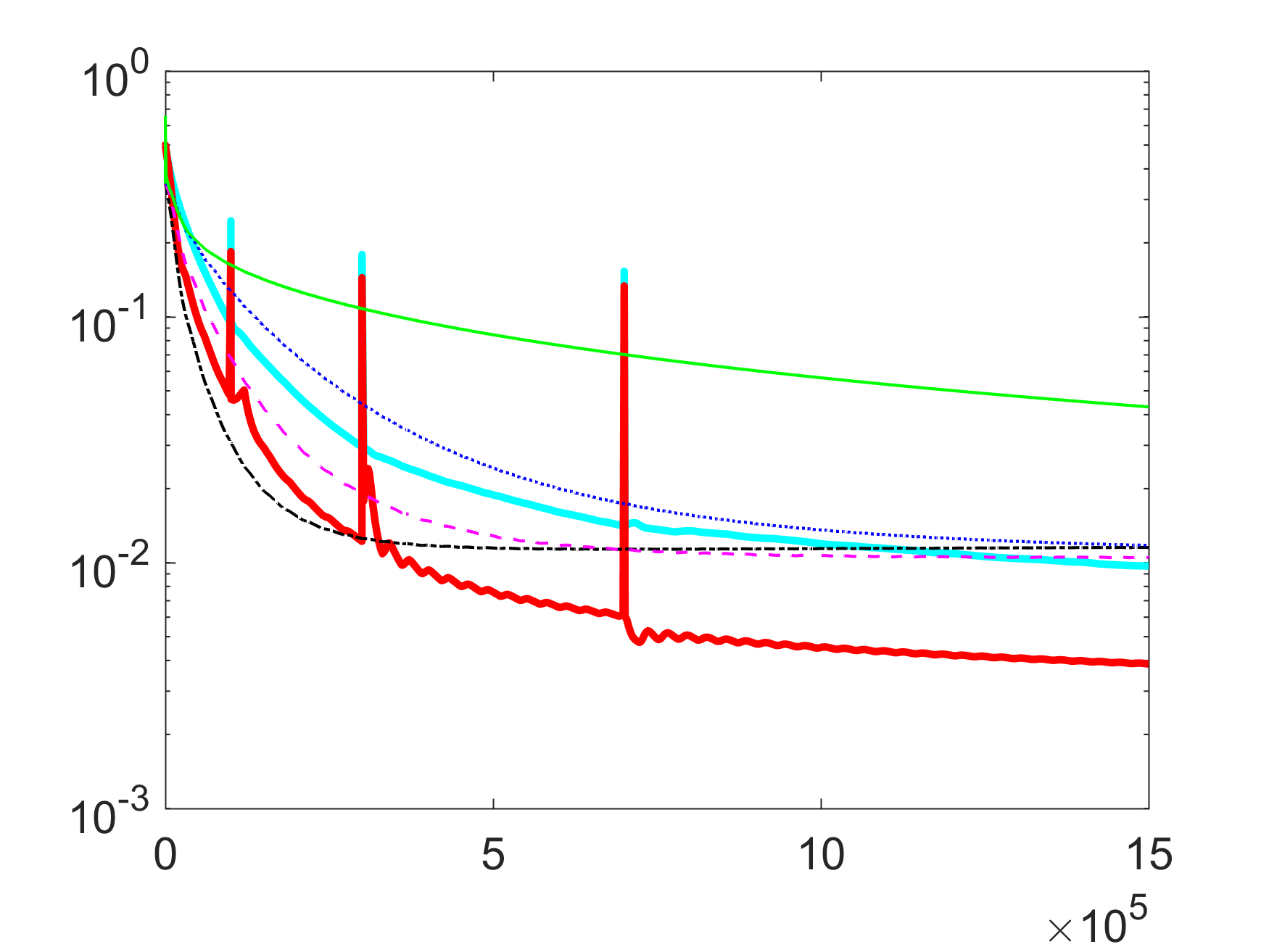}}
		\\
		{} 
		& { Number of iterations}
	\end{tabular}
	\caption{{\color{black}Performance} comparison for the distributed case with $N=10$. Our algorithm with stepsize $\eta = 0.05$ and $K=4$ ({\large\tc{red}{$\textbf{--\!--}$}}) {\color{black}achieves a smaller} optimality gap than stochastic primal-dual algorithm with stepsize: $\eta=0.05$ ($\text{-}$$\cdot$$\text{-}$), $\eta=0.025$ ({\color{magenta}$\text{-}\text{-}\text{-}$}), $\eta=0.0125$ ({\color{black}$\cdot$$\cdot$$\cdot$$\cdot$}), and $\eta = 1/\sqrt{t}$ ({\color{green}$\text{--\!--}$}).  {\color{black} It also provides a smaller optimality gap than the approach that utilizes pre-collected} iid samples in a buffer  ({\large\color{figcolor1}$\textbf{--\!--}$}).  }
	\label{fig.compare100}
\end{figure}

\begin{figure}[]
	\centering
	\begin{tabular}{cc}
		\rotatebox{90}{\quad\quad\quad\quad\quad\quad\quad\small Optimality gap} 
		&
		\hspace*{-0.3cm}  
		{\label{} \includegraphics[scale=0.6]{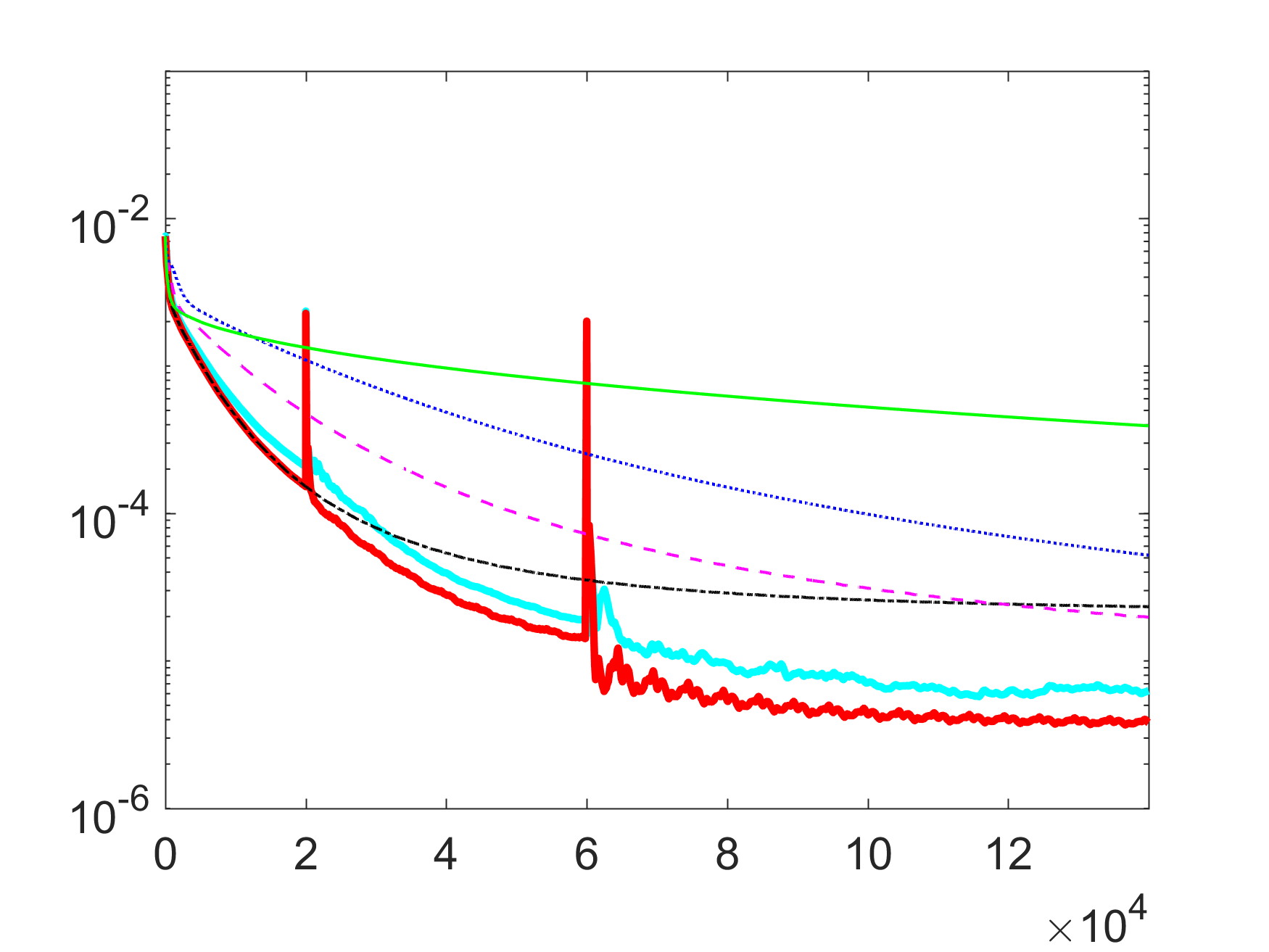}}
		\\
		{} 
		& { Number of iterations}
	\end{tabular}
	\caption{ {\color{black}Performance} comparison for the distributed case with $N=5$. Our algorithm with stepsize $\eta = 0.1$ and $K=3$ ({\large\color{red}$\textbf{--\!--}$}) {\color{black}achieves a smaller} optimality gap than stochastic primal-dual algorithm with stepsize: $\eta=0.5$ ($\text{-}$$\cdot$$\text{-}$), $\eta=0.25$ ({\color{magenta}$\text{-}\text{-}\text{-}$}), $\eta=0.125$ ({\color{black}$\cdot$$\cdot$$\cdot$$\cdot$}), and $\eta = 1/\sqrt{t}$ ({\color{green}$\text{--\!--}$}). {\color{black} It also provides a smaller optimality gap than the approach that utilizes pre-collected} iid samples in a buffer  ({\large\color{figcolor1}$\textbf{--\!--}$}).  }
	\label{fig.mmdps}
\end{figure}

	\vspace*{-2ex}
\section{Concluding remarks}\label{sec.concl}

In this paper, we formulate the multi-agent temporal-difference learning as a distributed primal-dual stochastic saddle point problem. We propose a new \tc{black}{online} distributed homotopy-based primal-dual algorithm \tc{black}{for minimizing the mean square projected Bellman error under the Markovian setting and establish an $O(1/T)$ error bound}. Our result improves the best known $O(1/\sqrt T)$ \tc{black}{error bound} for general stochastic primal-dual algorithms and it demonstrates that distributed saddle point programs can be solved \tc{black}{efficiently even in applications with limited time budgets}. 

Our results add to the growing body of literature on multi-agent reinforcement learning. The developed framework motivates further study including: 
(i) Our distributed policy evaluation is based on the linear function approximation. It is of interest to study nonlinear function approximators, e.g., neural networks. \tc{black}{Convex-concave property of the primal-dual formulation no longer holds in this setup and different analysis is required. (ii) It is of interest to examine transient performance of a distributed  algorithms that, instead of our consensus-based approach, utilize a network diffusion strategy~\cite{sayed2014adaptation}.} (iii) Our algorithm requires synchronous communication over a network with doubly stochastic matrices. This can be restrictive in applications \tc{black}{that involve directed networks}, communication delays, and time-varying channels. It is thus relevant to examine the effect of alternative communication schemes. (iv) Our multi-agent MDP assumes that agents function without any failures. In applications, it is of interest to examine a setup in which agents may experience communication/computation failures with some of them acting maliciously. 

	\vspace*{-2ex}
\appendix

	\vspace*{-1ex}
\subsection{Proof of Lemma~\ref{lem.netavearaging}}
\label{sec.proof.netavearaging}

{\color{black}We begin with the triangle inequality, 
$
\Vert x_{j,k}(t) -\bar{x}_k(t) \Vert 
\leq
\Vert x_{j,k}(t) -\underbar{$x$}_k(t) \Vert + \Vert \underbar{$x$}_k(t) - \bar{x}_k(t) \Vert
$ 
where 
$\underbar{$x$}_k(t) = P_{\mc X}(\bar{x}_k'(t))$ and $\bar{x}_k(t) = \frac{1}{N}\sum_{j\,=\,1}^{N}{x}_{j,k}(t))$.
First, the non-expansiveness of projection $P_{\mc X}$ implies that 
\begin{equation}\label{eq.decomp0}
\Vert x_{j,k}(t) -\underbar{$x$}_k(t) \Vert 
\, \leq \, 
\Vert x'_{j,k}(t) -\bar{x}'_k(t) \Vert
\end{equation}
where $x_{j,k}(t) = P_{\mc X}(x_{j,k}'(t))$.
Second, by the convexity of norm $\Vert\cdot\Vert$ and non-expansiveness of projection $P_{\mc X}$, 
\begin{equation}\label{eq.decomp1}
\Vert \underbar{$x$}_k(t) - \bar{x}_k(t) \Vert
\, \leq \,
\frac{1}{N}\sum_{j\,=\,1}^{N} \Vert  P_{\mc X}(\bar{x}_k'(t)) - P_{\mc X}(x_{j,k}'(t)) \Vert 
\, \leq \,
\frac{1}{N}\sum_{j\,=\,1}^{N}  \Vert x'_{j,k}(t) -\bar{x}'_k(t) \Vert. 
\end{equation}
}
Next, we focus on $ x'_{j,k}(t)$ and $\bar{x}'_k(t)$.
Let $[W^s]_j$ be the $j$th row of $W^s$ and $[W^s]_{ji}$ be the $(j,i)$th element of $W^s$. 
For any $t\geq 2$, the primal update  $x_{j,k}'(t) $ of Algorithm~\ref{alg.main} can be expressed as 
\be
\ba{rcl}\label{eq:recursion}
x_{j,k}'(t) &\!\! =\!\! &\displaystyle\sum_{i\,=\,1}^N [W^{t-1}]_{ij} \,x_{i,k}'(1) 
\\[0.2cm]
&&\!\!\!\!\!\!\!\!\!\!\!\!\!\!\!\!\!\!\!\!\!\!\!\!\!\!\!\!\! -\, \eta_k \displaystyle\sum_{s\,=\,2}^{t-1}\sum_{i\,=\,1}^N[W^{t-s+1}]_{ij}\,
G_{i,x}(z_{i,k}(s-1);\xi_{k,s-1})  
\\[0.2cm]
&&\!\!\!\!\!\!\!\!\!\!\!\!\!\!\!\!\!\!\!\!\!\!\!\!\!\!\!\!\! -\, \eta_k\, G_{j,x}(z_{j,k}(t-1);\xi_{k,t-1})
\ea
\ee
and $x_{j,k}'(2) = \sum_{i=1}^N[W]_{ij}x_{i,k}'(1) - \eta_k G_{j,x}(z_{j,k}(1);\xi_{k,1})$. 
Similar to the argument of~\cite[Eqs. (26) and (27)]{duchi2012dual}, we utilize the gradient boundedness to bound the difference of~\eqref{eq:simple-update} and~\eqref{eq:recursion} by
\be\label{eq:indi-diff-bound}
\ba{rcl}
\!\!\!\!\l\Vert x_{j,k}'(t) \!-\! \bar{x}_k'(t) \r\Vert
&\!\! \leq\!\! & \displaystyle \sum_{i\,=\,1}^N\l\vert \frac1N - [W^{t-1}]_{ij}\r\vert  \Vert x'_{i,k}(1)\Vert  
\\[0.2cm]
&& \!\!\!\!\!\!\!\!\!\!\!\!\!\!\!\!\!\!\!\!\!\!\!\!\!\!\!\!\!\! +\,
\eta_k {\color{black}c} \displaystyle\sum_{s\,=\,2}^{t-1}\l\Vert \dfrac1N\mathbf{1} - [W^{t-s+1}]_j \r\Vert_1 \,+\, 2\,\eta_k {\color{black}c}.
\ea
\ee
Application of the Markov chain property of mixing matrix~\cite{duchi2012dual} on the second sum in~\eqref{eq:indi-diff-bound}  yields,
\begin{equation}\label{eq:mixing-bound}
\sum_{s\,=\,2}^{t-1}\l\Vert \frac1N\mathbf{1} - [W^{t-s+1}]_j \r\Vert_1 
\,\leq\,
 \frac{2\log(\sqrt{N}\,T_k)}{1-\sigma_2(W)}.
\end{equation}
For $x_{i,k}'(1)=\frac{1}{T_{k-1}}\sum_{t=1}^{T_{k-1}}x_{i,k-1}(t)$ where $x_{i,k-1}(t) = P_{\mc X}(x_{i,k-1}'(t))$ and $0\in\mc X$, we utilize the non-expansiveness of projection to bound it as
\[
\Vert x_{i,k}'(1)\Vert  
\,\leq\, 
\dfrac{1}{T_{k-1}}\displaystyle\sum_{t\,=\,1}^{T_{k-1}} \l\Vert x_{i,k-1}'(t) \r\Vert.
\]
Using~\eqref{eq:recursion} at round $k-1$, we utilize the property of doubly stochastic $W$  to have  
\begin{equation*}\label{eq:intermediate-bound}
\Vert x_{i,k-1}'(t)\Vert 
\,\leq\,
 \displaystyle\sum_{j\,=\,1}^N [W^{t-1}]_{ji}\,\Vert x_{j,k-1}'(1)\Vert \,+\, 2 \eta_{k-1}T_{k-1} {\color{black}c}.
\end{equation*}
Repeating this inequality for $k-2, k-3, \ldots,1$ yields,
\begin{equation}
\label{eq:zero-bound}
\Vert x_{i,k}'(1)\Vert 
\,\leq\,
 2\displaystyle \sum_{l\,=\,1}^{k-1}\eta_l \,T_l\, {\color{black}c}.
\end{equation}
where we use $x_{j,1}'(1) = 0$ for all $j\in\mc V$.

Now, we are ready to show the desired result. 
Notice that $  \Vert x'_{j,k}(1) - \bar{x}'_k(1)\Vert \leq  \Vert x'_{j,k}(1)\Vert + \frac{1}{N}\sum_{j\,=\,1}^{N}  \Vert x'_{j,k}(1)\Vert $. 
We collect~\eqref{eq:mixing-bound} and~\eqref{eq:zero-bound} for~\eqref{eq:indi-diff-bound}, and average it over $t=1,\ldots,T_k$ to obtain,
\[
\begin{array}{rcl}
&&\!\!\!\!\!\!\!\!\!\! \dfrac{1}{T_k} \displaystyle\sum_{t\,=\,1}^{T_k} \Vert x'_{j,k}(t) - \bar{x}'_k(t)\Vert
\\[0.2cm]
& \!\! =\!\!  & 
\dfrac{1}{T_k} \displaystyle\sum_{t\,=\,2}^{T_k} \Vert x'_{j,k}(t) - \bar{x}'_k(t)\Vert \,+\, \dfrac{1}{T_k} \Vert x'_{j,k}(1) - \bar{x}'_k(1)\Vert 
\\[0.2cm]
&\!\! \leq\!\! &  
\dfrac{2\eta_k{\color{black}c}\log(\sqrt{N}T_k)}{1-\sigma_2(W)} 
 \,+\, \dfrac{4}{T_k}\displaystyle\sum_{l\,=\,1}^{k-1}\eta_lT_{l}{\color{black}c} 
 \,+\, 2\eta_k {\color{black}c}
\\[0.2cm]
&&
\,+\,  \dfrac{2{\color{black}c}}{T_k}\displaystyle\sum_{t\,=\,2}^{T_k}\sum_{i\,=\,1}^N \l\vert \dfrac1N - [W^{t-1}]_{ji}\r\vert \sum_{l\,=\,1}^{k-1}\eta_l T_{l}.
\end{array}
\]
Bounding the sum $\sum_{t\,=\,2}^{T_k}\sum_{i\,=\,1}^N\l\vert \frac1N - [W^{t-1}]_{ji}\r\vert $ by~\eqref{eq:mixing-bound} and application of~\eqref{eq.decomp0} and~\eqref{eq.decomp1} complete the proof. 

	\vspace*{-1ex}
\subsection{Martingale concentration bound}
\label{app.A}

We state a useful result about martingale sequence.

\begin{lemma}\label{lem:martingale}
	Let $\{X(t)\}_{t=1}^{T}$ be a martingale difference sequence in $\mb R^d$, and let $\Vert X(t)\Vert\leq M$. Then, 
	\begin{equation}\label{eq:martingale}
	\expect{\,\l\Vert \dfrac{1}{T}\displaystyle\sum_{t\,=\,1}^{T}X(t)\r\Vert^2\,} 
	\,\leq\, 
	\dfrac{4M^2}{T}.
	\end{equation}
\end{lemma}

\begin{proof}
We recall the classic concentration result in~\cite{pinelis1994optimum}: for any $\delta\geq0$, we have
\[
P\l(\,\l\Vert \dfrac{1}{T}\sum_{t\,=\,1}^{T}X(t)\r\Vert^2 \, \geq \, \frac{4M^2\delta}{T}\,\r) \,\leq\, {\rm e}^{-\delta}.
\]

The left-hand side of~\eqref{eq:martingale} can be expanded into 
\begin{equation*}\label{lem:martingale.expectation}
\begin{array}{rcl}
&& \!\!\!\!\!\!\!\!\! \displaystyle\int_0^\infty P\l(\,\l\Vert \dfrac{1}{T}\sum_{t\,=\,1}^{T}X(t)\r\Vert^2\,\geq\, s\,\r)ds
\\[0.2cm]
& \!\! =\!\! & \dfrac{4M^2}{T}\displaystyle\int_0^\infty P\l(\l\Vert \dfrac{1}{T}\sum_{t=1}^{T}X(t)\r\Vert^2 \,\geq \,\frac{4M^2\delta}{T}\r)d\delta
\\[0.2cm]
&\!\! \leq\!\! &\dfrac{4M^2}{T}\displaystyle\int_0^\infty {\rm e}^{-\delta}d\delta
\; \leq \; \dfrac{4M^2}{T}.
\end{array}
\end{equation*}
\end{proof}

	\vspace*{-1ex}
\subsection{Proof of Lemma~\ref{lem:objective2PD}}
\label{sec.proof.objective2PD}

It is clear that $\text{Err}(\hat{x}_{i,k})\geq 0$ from the optimality of $x^\star$ in~\eqref{eq:objgap}. 
The optimality of $y_j^\star $ yields,
\[
\frac1N\sum_{j\,=\,1}^N\, \psi_j(x^\star,y_j^\star)
\,\geq\,
\frac1N\sum_{j\,=\,1}^N\, \psi_j(x^\star,\hat y_{j,k}).
\]
Thus, using~\eqref{eq:objgap1} and~\eqref{eq.surrogate}, we have    
$\text{Err}(\hat{x}_{i,k}) \leq \text{Err}'(\hat{x}_{i,k},\hat{y}_k)$.

	\vspace*{-2ex}
\subsection{Proof of Lemma~\ref{lem:convexconcave}}
\label{sec.proof.convexconcave}

{\color{black}To show~\eqref{eq.cc1}, we apply the strong convexity of $f_j(x)$,
	\[
	\text{Err}(\hat{x}_{i,k}) 
	\, \geq \,
	\frac{1}{N} \sum_{j\,=\,1}^{N} \dotp{\nabla f_j(x^\star)} {\hat x_{i,k} - x^\star}
	+
	\frac{L_x}{2} \Vert \hat x_{i,k} - x^\star \Vert^2
	\, \geq \,
	\frac{L_x}{2} \Vert \hat x_{i,k} - x^\star \Vert^2
	\]
	where the second inequality is due to the optimality of $x^\star$.}

To show~\eqref{eq.cc2}, we subtract and add $ \frac{1}{N}\sum_{j=1}^{N} \psi_j(x^\star,y_{j}^\star)$ into~\eqref{eq.surrogate} and apply the strong concavity of $\psi_j(x,y_{j})$ in $y_j$,
\begin{equation*}
\begin{array}{rcl}
\text{Err}'(\hat{x}_{i,k},\hat{y}_k)
& \!\! =\!\! & \text{Err}(\hat{x}_{i,k}) \,+\, \dfrac{1}{N}\!\displaystyle\sum_{j\,=\,1}^{N}\! (\psi_j(x^\star,y_{j}^\star)
-
\psi_j(x^\star,\hat{y}_{j,k}) )
\\[0.2cm]
& \!\! \geq\!\! & \dfrac{1}{N}\displaystyle\sum_{j\,=\,1}^{N} (\,\psi_j(x^\star,y_{j}^\star\,)
-
\psi_j(x^\star,\hat{y}_{j,k}) )
\\[0.2cm]
&\!\! \geq\!\! & \dfrac{L_y}{2N} \displaystyle\sum_{j\,=\,1}^{N}\Vert y_j^\star-\hat{y}_{j,k} \Vert^2.
\end{array}
\end{equation*}
where 
$\text{Err}(\hat{x}_{i,k})\geq 0$ is omitted in the first inequality, and {\color{black}the second inequality follows the strong concavity assumption of $\psi_j(x^\star,y_j)$ in $y_j$ and the optimality of $y_j^\star = \argmax_{y_j\,\in\,\mc Y} \psi_j(x^\star,y_j)$.}

To show~\eqref{eq.cc3}, because of the optimality of $x^\star$ and $ y_{j}^\star$, we first have  
$
\psi_j(\hat{x}_{i,k},y_{j}^\star)
-
\psi_j(x^\star,y_{j}^\star)
\geq 0
$, and 
$\psi_j(x^\star,y_{j}^\star)
-
\psi_j(x^\star,\hat{y}_{j,k})
\geq0$.
This further shows that 
$\psi_j(\hat{x}_{i,k},\hat y_{j,k}^\star)-\psi_j(x^\star,\hat{y}_{j,k}) 
=
(\psi_j(\hat{x}_{i,k},\hat y_{j,k}^\star)
-
\psi_j(\hat{x}_{i,k},y_{j}^\star))
+
(\psi_j(\hat{x}_{i,k},y_{j}^\star)
-
\psi_j(x^\star,y_{j}^\star))
+
(\psi_j(x^\star,y_{j}^\star)
-
\psi_j(x^\star,\hat{y}_{j,k}))
\geq
\psi_j(\hat{x}_{i,k},\hat y_{j,k}^\star)
-
\psi_j(\hat{x}_{i,k},y_{j}^\star)
$.
Hence, we have
\begin{equation*}
\begin{array}{rcl}
\text{Err}'(\hat{x}_{i,k},\hat{y}_k)
& \!\! \geq\!\! & \dfrac{1}{N}\displaystyle\sum_{j\,=\,1}^{N}\, (\,\psi_j(\hat{x}_{i,k},\hat y_{j,k}^\star)
-
\psi_j(\hat{x}_{i,k},y_{j}^\star)\,)
\\[0.2cm]
&\!\! \geq\!\! & \dfrac{L_y}{2N} \displaystyle\sum_{j\,=\,1}^{N}\Vert y_j^\star-\hat{y}_{j,k}^\star \Vert^2
\end{array}
\end{equation*}
where {\color{black}the second inequality follows from the strong concavity assumption of $\psi_j(\hat x_{i,k},y_j)$ in $y_j$ and the optimality of $\hat y_{j,k}^\star = \argmax_{y_j\,\in\,\mc Y} \psi_j(\hat x_{i,k}^\star,y_j)$.}





%
	\vspace*{-2ex}
	\renewcommand{\baselinestretch}{1.25} 

\end{document}